\documentclass[11pt]{amsart}
\usepackage{amscd,amssymb}
\usepackage[all]{xy}
\usepackage[colorlinks,plainpages,backref,urlcolor=blue]{hyperref}

\newtheorem{theorem}{Theorem}[section]
\newtheorem{corollary}[theorem]{Corollary}
\newtheorem{lemma}[theorem]{Lemma}
\newtheorem{prop}[theorem]{Proposition}

\theoremstyle{definition}
\newtheorem{definition}[theorem]{Definition}
\newtheorem{example}[theorem]{Example}
\newtheorem{remark}[theorem]{Remark}

\newtheorem*{ack}{Acknowledgments}

\newenvironment{romenum}
{

\begin{enumerate}}{\end{enumerate}}

\newenvironment{alphenum}
{

\begin{enumerate}}{\end{enumerate}}

\newcommand{\Z}{\mathbb{Z}}
\newcommand{\Q}{\mathbb{Q}}
\newcommand{\R}{\mathbb{R}}
\newcommand{\C}{\mathbb{C}}

\newcommand{\RP}{\mathbb{RP}}
\newcommand{\T}{\mathbb{T}}
\renewcommand{\k}{\Bbbk}

\DeclareMathOperator{\rank}{rank}

\DeclareMathOperator{\im}{im}

\DeclareMathOperator{\id}{id}
\DeclareMathOperator{\ab}{{ab}}
\DeclareMathOperator{\abf}{{abf}}
\DeclareMathOperator{\Sym}{Sym}
\DeclareMathOperator{\ch}{char}

\DeclareMathOperator{\Hom}{{Hom}}
\DeclareMathOperator{\Tor}{{Tor}}

\DeclareMathOperator{\ann}{{ann}}

\DeclareMathOperator{\ev}{ev}

\DeclareMathOperator{\supp}{{supp}}
\DeclareMathOperator{\lk}{lk}
\DeclareMathOperator{\Tors}{{Tors}}
\DeclareMathOperator{\FF}{{F}}
\DeclareMathOperator{\FP}{{FP}}
\DeclareMathOperator{\df}{{def}}

\DeclareMathOperator{\init}{in}
\DeclareMathOperator{\Grass}{Gr}

\newcommand{\RR}{\mathcal{R}}
\newcommand{\VV}{\mathcal{V}}
\newcommand{\WW}{\mathcal{W}}

\newcommand{\A}{\mathcal{A}}

\newcommand{\G}{\Gamma}

\newcommand{\V}{\mathsf{V}}

\newcommand{\sV}{\mathsf{V}}
\newcommand{\sW}{\mathsf{W}}

\newcommand{\h}{\mathfrak{h}}

\newcommand{\m}{\mathfrak{m}}
\newcommand{\bb}{\mathfrak{b}}

\newcommand{\wX}{\widetilde{X}}

\newcommand{\SR}{\k\langle L\rangle}
\newcommand{\compl}{\scriptscriptstyle{\complement}}
\newcommand{\rat}{\mathcal{R}\Gamma_{\iota}}

\newcommand{\same}{\Longleftrightarrow}
\newcommand{\surj}{\twoheadrightarrow}
\newcommand{\inj}{\hookrightarrow}
\newcommand{\isom}{\xrightarrow{\,\simeq\,}}
\def\set#1{{\left\{#1\right\}}}
\newcommand{\abs}[1]{\left| #1 \right|}

\begin{document}
%\date{August 28, 2009}

\title[$\Sigma$-invariants and homology jumping loci]{%
Bieri--Neumann--Strebel--Renz invariants and\\ homology jumping loci}

\author[Stefan Papadima]{Stefan Papadima$^1$}
\address{Institute of Mathematics Simion Stoilow, 
P.O. Box 1-764,
RO-014700 Bucharest, Romania}
\email{Stefan.Papadima@imar.ro}
\thanks{$^1$Partially supported by grant 
CNCSIS ID-1189/2009-2011 of the 
Romanian Ministry of Education and Research}

\author[Alexander~I.~Suciu]{Alexander~I.~Suciu$^2$}
\address{Department of Mathematics,
Northeastern University,
Boston, MA 02115, USA}
\email{a.suciu@neu.edu}
\thanks{$^2$Partially supported by NSA grant H98230-09-1-0012, 
and an ENHANCE grant from Northeastern University}



\begin{abstract}
We investigate the relationship between the geometric 
Bieri--Neumann--Strebel--Renz invariants of a space 
(or of a group), and the jump loci for homology with 
coefficients in rank $1$ local systems over a field.  
We give computable upper bounds for the geometric 
invariants, in terms of the exponential tangent cones to 
the jump loci over the complex numbers. Under suitable 
hypotheses, these bounds can be expressed in terms of 
simpler data, for instance, the resonance varieties associated 
to the cohomology ring. These techniques yield 
information on the homological finiteness properties 
of free abelian covers of a given space, 
and of normal subgroups with abelian quotients 
of a given group. We illustrate our results in a variety 
of geometric and topological contexts, such as toric 
complexes and Artin kernels, as well as K\"{a}hler and 
quasi-K\"{a}hler manifolds. 
\end{abstract}
\maketitle

\tableofcontents

\section{Introduction}
\label{sect:intro}

\subsection{$\Sigma$--invariants of groups and spaces}
\label{intro:bns}
In their landmark paper \cite{BNS}, Bieri, Neumann, 
and Strebel associated to every finitely generated group $G$ 
an  open, conical subset  $\Sigma^1(G)$ of the real 
vector space $\Hom(G,\R)$. The BNS invariant and 
its higher-order generalizations, the invariants 
$\Sigma^r(G,\Z)$ of Bieri and Renz \cite{BR}, hold subtle 
information about the homological finiteness properties 
of normal subgroups of $G$ with abelian quotients. 
The actual computation of the $\Sigma$-invariants 
is enormously complicated, and has been achieved 
so far only for some special classes of groups, such as
metabelian groups \cite{BGr}, one-relator groups \cite{Br}, 
and right-angled Artin groups \cite{MMV}.

The definition of the Bieri--Neumann--Strebel--Renz invariants 
was further extended to connected, finite CW-complexes by 
Farber, Geoghegan, and Sch\"{u}tz \cite{FGS}. Slightly more 
generally, let $X$ be a connected CW-complex with finite 
$1$-skeleton, and fundamental group $G$.  Based on a 
key result of Sikorav \cite{Si}, we will say that a nontrivial 
character $\chi \in \Hom (G, \R)$ belongs to $\Sigma^q(X, \Z)$, 
for $q\ge 0$, if the twisted homology of $X$ with coefficients 
in the Novikov--Sikorav completion of the group ring,
$\widehat{\Z G}_{-\chi}$, vanishes up to degree $q$.  
If the group $G$ is of type $\FF_k$, i.e., admits a classifying 
space $K(G,1)$ with finite $k$-skeleton, then 
$\Sigma^q(G, \Z)=\Sigma^q(K(G,1), \Z)$, for all $q\le k$. See
\S\S\ref{subsec:nov}--\ref{subsec:defspaces} for more details.

In this paper, we prove that each BNSR invariant $\Sigma^q(X, \Z)$ 
is contained in the complement of a rationally defined subspace 
arrangement in the vector space $\Hom (\pi_1(X), \R)$. Our 
upper bounds for the $\Sigma$-invariants of spaces and groups 
are effectively computable, and, in many cases, sharp.  The 
bounds are derived from the characteristic varieties, defined 
as the jump loci for the homology of $X$ with coefficients 
in rank $1$ local systems. Under a certain hypothesis, 
these upper bounds can be expressed in terms of even 
simpler data, namely, the resonance varieties associated 
to the cohomology ring of $X$. 

To obtain our results, we make essential use of the exponential 
tangent cone construction, introduced in our joint work with 
A.~Dimca \cite{DPS-jump}. Another important ingredient is 
the relationship between the characteristic varieties, 
their (exponential) tangent cones at the origin, and 
the resonance varieties, also developed in \cite{DPS-jump}. 

\subsection{Fibrations, {A}lexander polynomial, and {T}hurston norm}
\label{intro:alex thurston}

Before describing our work in more detail, let us briefly review 
some of the topological context, starting with the following 
question:  When does a compact, connected manifold $M^n$ 
fiber over the circle? In dimension $3$,  Stallings \cite{St1} 
reduced this question to a purely group-theoretical one:  
If $\pi_1(M^3)$ admits a surjective homomorphism to $\Z$ 
with finitely generated kernel $N$, then $M^3$ fibers over 
$S^1$, with the fiber having fundamental group $N$. 

This striking result led to a follow-up question: Given a finitely 
generated group $G$, and an epimorphism $\nu\colon G\surj \Z$, 
what are the finiteness properties of $N=\ker(\nu)$?  It was 
Stallings again who showed in \cite{St2} that both $G$ 
and $N$ could be finitely presented, and yet $H_3(N,\Z)$ 
fails to be finitely generated; in particular, such $N$ is not of 
type $\FP_3$. (The finiteness properties of spaces and groups 
are reviewed in \S\ref{sec:finite}.)

An old observation in knot theory is that the Alexander 
polynomial $\Delta_K(t)$ of a fibered knot $K$ is monic.  
More generally, Dwyer and Freed \cite{DF} showed that 
the support varieties of the Alexander invariants of a 
finite CW-complex completely determine the homological 
finiteness properties of its free abelian covers.

For each compact, orientable $3$-manifold $M$, Thurston \cite{Th} 
defined a (pseudo-)norm on $H^1(M,\R)$, which measures the 
minimal complexity of an embedded surface dual to a given 
cohomology class. In \cite{BNS}, Bieri, Neumann, and Strebel 
found a remarkable connection between the Thurston norm of 
$M$ and the BNS invariant of $G=\pi_1(M)$, thereby recasting  
Stallings' fibration theorem in a more general framework. 

\subsection{Homology jumping loci and tangent cones}
\label{intro:cjl}

We now return to one of the central objects 
of our study. Let $X$ be a connected CW-complex 
with finitely many cells up to dimension $k$, for some 
$k\ge 1$.  Let $G=\pi_1(X)$, and fix a coefficient field $\k$. 
Each character $\rho\colon G\to \k^{\times}$  gives rise 
to a rank~$1$ local system on $X$, call it $\k_{\rho}$.  
The {\em characteristic varieties} of $X$ are Zariski 
closed subsets of the algebraic group $\Hom(G,\k^{\times})$, 
defined as follows.   For each $0\le i\le k$ and $d>0$, the 
variety $\VV^i_d(X,\k)$ consists of those characters $\rho$ 
for which the dimension of $H_i(X, \k_{\rho})$ is at least $d$.  
As we show in \S\ref{sec:cjl}, these varieties are closely 
related to the support loci of the elementary ideals of the 
Alexander invariants $H_i(X^{\ab},\k)$, where $X^{\ab}$ 
is the maximal abelian cover of $X$.  

In \S\ref{sec:tcone}, we discuss two types of approximations 
to the variety $W=\VV^i_d(X,\C)$ around the trivial character 
$1\in \Hom (G, \C^{\times})$.  One is the usual tangent 
cone, $TC_1(W)$, while the other is the {\em exponential 
tangent cone}, $\tau_1(W)$, which consists of those classes 
$z\in H^1(X,\C)$ for which the curve $\set{\exp(tz)\mid t\in\C}$ 
is completely contained in $W$. It turns out that $\tau_1(W)$ 
is a finite union of rationally defined linear subspaces, all 
contained in $TC_1(W)$.  In general, this inclusion is strict, but, 
if all irreducible components of $W$ containing $1$ are subtori, 
then $\tau_1(W)= TC_1(W)$. 

We turn in \S\ref{sec:res var} to another type of jumping loci, 
depending only on the cohomology algebra $A=H^*(X,\k)$.  
If $\ch \k=2$, assume $H_1(X,\Z)$ has no $2$-torsion.  Then, 
for each $z\in A^1$, right-multiplication by $z$ defines a 
cochain complex $(A,z)$.  The {\em resonance varieties} 
of $X$ are homogeneous subvarieties of the affine space 
$\Hom (G, \k)$:  for each $i$ and $d$ as above, the 
variety $\RR^i_d(X,\k)$ consists of those classes $z$ for 
which the dimension of $H^i(A,z)$ is at least $d$.

The tangent cone at $1$ to $\VV_d^i(X, \C)$ is contained in 
$\RR_d^i(X, \C)$ \cite{Li}, but, in general, the inclusion is strict 
\cite{MS, DPS-serre}. We say that the {\em exponential formula} 
holds for $X$ (in degree $i\le k$ and depth $d>0$) if the 
exponential map restricts to an isomorphism of analytic germs, 
$\exp\colon (\RR^i_d(X,\C), 0)\isom  (\VV^i_d(X,\C), 1)$.
In this case, $\tau_1(\VV^i_d(X,\C))=\RR^i_d(X,\C)$.

\subsection{Characteristic varieties and $\Sigma$-invariants}
\label{intro:tau1 bound}

For a variety $W\subseteq \Hom(G,\C^{\times})$, let 
$\tau_1^{\R}(W)\subseteq \Hom(G,\R)$ denote the 
real points on $\tau_1(W)$. For a subset $U \subseteq \Hom(G,\R)$,
let $U^{\compl}$ denote its complement.  Under some mild assumptions, 
we relate in \S\ref{sec:main}  the BNSR invariants and the homology 
jumping loci of a space $X$, as follows.  

\begin{theorem}
\label{intro:thm1}
Let $X$ be a connected CW-complex with finitely many cells 
up to dimension $k$, for some $k\ge 1$. Then, for each $q\le k$, 
the following ``exponential tangent cone upper bound" holds:
\begin{equation}
\label{intro-exptc}
\Sigma^q(X, \Z)\subseteq \Big( \bigcup_{i\le q} 
\tau_1^{\R}\big(\VV_1^i(X, \C)\big)\Big)^{\compl}.
\end{equation}

If, moreover, the exponential formula holds for $X$ up to 
degree $q$, and depth $d=1$, then the following 
``resonance upper bound" obtains:
\begin{equation}
\label{intro-res}
\Sigma^q(X, \Z)\subseteq \Big( \bigcup_{i\le q} 
\RR_1^i(X, \R)\Big)^{\compl}.
\end{equation}
\end{theorem}

When the group $G$ admits a classifying space $K(G, 1)$ satisfying 
the assumptions from Theorem \ref{intro:thm1}, one may replace 
in \eqref{intro-exptc} and \eqref{intro-res} above $X$ by $G$ 
on the left-hand side, and $X$ by $K(G,1)$ on the right-hand side.

This theorem answers a question raised by one of us at 
a Mini-Workshop held at Oberwolfach in August 2007, 
and recorded in the problems list of its proceedings \cite{FSY}. 

Clearly, every additive character $\chi\in \Hom (G, \R)$ factors 
as $\chi=\iota \circ \xi$, where $\xi\colon G\surj \Gamma$ is a 
surjection onto a lattice in $\R$, and $\iota\colon \Gamma\inj \R$ 
is the inclusion map.  Since the Novikov--Sikorav completion is a 
rather complicated ring, we turn to a commutative approximation:  
the Novikov completion, $\widehat{\Z\Gamma}_{\iota}$. This idea 
is certainly not new, as it has been used by Farber \cite[Chapter 1]{Fa} 
and Pajitnov \cite{Pa} in related contexts. The main novelty consists 
in using our exponential tangent cone formalism, to obtain a computable, 
conceptually simpler, upper bound for the $\Sigma$-invariants, 
in terms of homology with twisted (rank $1$) coefficients.

Our approach to the BNSR invariants via Novikov homology 
leads to further results, relating these invariants to the characteristic 
varieties over arbitrary fields.  In \S\ref{sect:arbch}, we prove the 
following theorem. 
 
\begin{theorem}
\label{intro:thm2}
Let $X$ be a connected CW-complex with finite 
$k$-skeleton ($k\ge 1$), and set $G=\pi_1(X)$. 
\begin{enumerate}
\item \label{d1}
Let $\rho\colon G\to \k^{\times}$ be a homomorphism  
such that $\rho \in \bigcup_{i\le q} \VV^i_1(X, \k)$, for some 
$q\le k$, and let $v\colon \k^{\times}\to \R$ be a 
valuation on $\k$ such that $v\circ \rho\ne 0$. 
Then $v\circ \rho \not\in \Sigma^q(X, \Z)$. 

\item \label{d2}
Let $\chi=\iota\circ \xi\colon G\to \R$ be an additive 
character such that $-\chi \in \Sigma^q(X, \Z)$, for 
some $q\le k$, and let $\rho\colon  \Gamma \to \k^{\times}$ 
be a character which is not an algebraic integer.
Then $\rho\circ \xi \not\in \bigcup_{i\le q} \VV^i_1(X, \k)$.
\end{enumerate}
\end{theorem}

Part \eqref{d1} of this theorem extends one of the main results 
from \cite{Dz}, proved by Delzant only for $\Sigma^1$, using  
a different approach.

\subsection{Homology of free abelian covers}
\label{intro:hom cov}

In \S\ref{sect:abel covers}, we refine the aforementioned result of 
Dwyer and Freed \cite{DF}, as follows. 
Let $X$ be connected CW-complex with finite $k$-skeleton 
($k\ge 1$).  Let us say $X$ satisfies condition $(*)_k$ if, for each 
$i\le k$, the part of $\VV^i_1(X, \C)$ contained in the connected 
component of $1\in \Hom(\pi_1(X), \C^{\times})$ is a union of 
subtori passing through $1$, and $TC_1(\VV^i_1(X,\C))=\RR^i_1(X,\C)$. 

Now let $\nu\colon  \pi_1(X)\surj \Z^r$ ($r\ge 1$) be an epimorphism, 
and $X^{\nu}\to X$ the corresponding Galois cover. If $r=1$, 
denote by $\nu_{\C}\in H^1(X,\C)$ the  corresponding cohomology class. 

\begin{theorem}
\label{intro:thm3} 
With notation as above, 
\begin{enumerate}
\item \label{free1}
If $r=1$, then 
$\bigoplus_{i\le k} H_{i} (X^{\nu}, \C)$ is finite-dimensional  
if and only if $\nu_{\C}$ does not belong to any of the 
linear subspaces comprising 
$\bigcup_{i\le k} \tau_1\big(\VV^i_1(X,\C))$.  

\item \label{free2}
If $r\ge 1$, and $X$ satisfies $(*)_k$, then 
$\bigoplus_{i\le k} H_{i} (X^{\nu}, \C)$ is finite-dimensional  
if and only if the pullback of $H^1(\Z^r, \C)$ along $\nu^*$ 
intersects $\bigcup_{i\le k} \RR^i_1(X,\C)$ only at $0$.
\end{enumerate}
\end{theorem}

As a consequence, homological finiteness of regular, free abelian 
covers is an open condition in the appropriate Grassmanian, 
provided the above assumptions on the space $X$ are met.  
On the other hand, if $r>1$, and condition $(*)_k$ is violated, 
the finiteness property is no longer an open condition, as 
shown by an example from \cite{DF}. 

\subsection{First applications}
\label{intro:apps}

For the rest of the paper, we illustrate our results---in particular, 
bounds \eqref{intro-exptc} and  \eqref{intro-res}  from 
Theorem \ref{intro:thm1}---with examples arising in a 
variety of geometric and topological contexts.  
The general inclusion \eqref{intro-exptc} 
may well be strict, simply because there are situations when 
$\Sigma^1\ne - \Sigma^1$. Nevertheless, \eqref{intro-exptc} 
can often be improved to better upper bounds. 

We start by discussing such an improvement in \S\ref{sec:apps}.  
It is shown in \cite{DPS-jump} that the exponential formula 
holds for $X$, in degree $i=1$ and arbitrary depth $d$, 
whenever $G=\pi_1(X)$ is a {\em $1$-formal}\/ group, 
i.e., its Malcev Lie algebra is quadratic.  In this case, \eqref{intro-res}
implies that $\Sigma^1 (G)\subseteq \RR^1_1(X, \R)^{\compl}$, with 
equality when $X$ belongs to a large class of compact K\"{a}hler 
manifolds (see \S\ref{intro:kahler} below).  

In \S\ref{sec:nilp}, we consider the class of finitely 
generated nilpotent groups.   If $G$ is such a group, we note 
that bound \eqref{intro-exptc} is sharp, that is, 
$\Sigma^q(G, \Z)= \big( \bigcup_{i\le q} 
\tau_1^{\R}(\VV_1^i(K(G, 1), \C))\big)^{\compl}$, 
for all $q\ge 0$.  On the other hand, bound \eqref{intro-res} 
fails for many nilpotent groups, such as the Heisenberg group. 

In \S\ref{sec:products}, we analyze the behavior 
of (co)homology jumping loci and $\Sigma$-invariants 
with respect to product and coproduct operations. 
In particular, we discuss the product formula for BNSR 
invariants of groups. Originally conjectured by Bieri \cite{Bi99}, 
this formula was recently proved by Bieri and Geoghegan \cite{BiG}. 
We show in Theorem \ref{thm:bound prod} that the class of groups of
type $\FF_k$ ($k\ge 1$) for which upper bound \eqref{intro-exptc} 
holds as an equality, for all $q\le k$, is closed under direct products. 

\subsection{Toric complexes}
\label{intro:toric}
Given a simplicial complex $L$ on $n$ vertices, the toric  
complex $T_L$ is the subcomplex of the $n$-torus 
obtained by deleting the cells corresponding to the 
missing simplices of $L$.  The fundamental group 
of $T_L$ is the right-angled Artin group $G_{\G}$ 
attached to the graph $\G$ defined as the $1$-skeleton 
of $L$.  A classifying space for $G_{\G}$ is the toric 
complex $T_{\Delta_{\G}}$, where $\Delta_{\G}$ is 
the flag complex of $\G$. 

In \S\ref{sec:artin}, we find a perfect complementary 
match between the BNSR invariants and the 
cohomology jumping loci of a right-angled Artin group.  
The $\Sigma$-invariants of such groups (with coefficients 
in an arbitrary commutative ring), were computed by 
Meier, Meinert, and VanWyk in \cite{MMV}, with a 
more convenient description given later on by Bux and Gonzalez 
in \cite{BG}.  The characteristic and resonance varieties of a toric 
complex $T_L$  were computed in \cite{PS-toric}, in terms of  the 
algebraic combinatorics of the underlying simplicial complex $L$. 
In particular, the exponential formula holds for a toric complex,
in arbitrary degree and depth. By Theorem \ref{intro:thm1}\eqref{intro-res},
\begin{equation}
\label{intro-toric}
\Sigma^k(T_L, \Z)\subseteq \Big( \bigcup_{i\le k} 
\RR_1^i(T_L, \R) \Big)^{\compl}, \quad \text{for all $k\ge 0$}.
\end{equation} 

Using all this information, we show in Corollary \ref{cor:res bns raag}
that
\begin{equation}
\label{eq=rmatch}
\Sigma^k(G_\G, \R)=\Big( \bigcup_{i\le k} 
\RR_1^i(T_{\Delta_{\G}}, \R) \Big)^{\compl},
\end{equation}
for all graphs $\G$, and for all $k\ge 0$. This equality recovers 
a result from \cite{PS-artin}, where we only dealt with the case 
$k=1$.  (That special case, it should be said, provided much 
of the initial impulse for this work.) Finally, we describe in 
Corollary \ref{cor:tfree} natural torsion-freeness assumptions 
on the integral homology of certain subcomplexes of
$\Delta_{\G}$, which force upper bound \eqref{intro-toric} 
to become an equality, in the case when $L= \Delta_{\G}$.

\subsection{Artin kernels}
\label{intro:ak}
In \S\ref{sec:toric cov}, we consider the regular, free abelian 
covers of a toric complex $T_L$.  Given a homomorphism 
$\chi\colon G_{\G}\surj \Z^r$, the homological finiteness 
properties over an arbitrary field $\k$ of the corresponding cover, 
$T_L^{\chi}$, depend only on the position of a certain linear 
subspace determined by $\chi$, with respect to a subspace 
arrangement determined by $L$.  This combinatorial dependence 
is made explicit in Theorem \ref{thm:toric cov}.  In characteristic 
$0$, things are even simpler: the vector space 
$\bigoplus_{i\le q} H_{i} (T_L^{\chi}, \Q)$ is finite-dimensional if 
and only if the pullback of $H^1(\Z^r,\Q)$ along $\chi^*$ intersects 
$\bigcup_{i\le q}\RR^i_1(T_L,\Q)$ only at $0$. 

Next, we focus on the Artin kernels 
$N_{\chi}=\ker(\chi\colon G_{\G}\surj \Z)$.  
Particularly interesting are the groups  $N_{\G}$ considered by 
Bestvina and Brady in \cite{BB}, corresponding to the ``diagonal" 
homomorphism $\nu$.  For example, if $\G$ is the $1$-skeleton 
of the octahedron, then $N_{\G}$ is none other than Stallings' group 
from \cite{St2}. In general, $N_{\G}$ is finitely generated if and only 
if $\G$ is connected, while $N_{\G}$ is finitely presented if and only 
if $\Delta_\G$ is simply-connected.

In previous work \cite{PS-toric}, we showed that the triviality of the 
monodromy action of $\Z$ on $\bigoplus_{i\le k} H_i(N_{\chi}, \Q)$ 
can be tested in a purely combinatorial way.  Moreover, if 
$H_1(N_{\chi}, \Q)$ is a trivial $\Q\Z$-module, then $N_{\chi}$ 
is finitely generated. Here, we compute in Theorem \ref{thm:cjl ak} 
the jump loci $\VV^1_1(N_{\chi}, \C)$ and $\RR^1_1(N_{\chi}, \C)$, 
assuming trivial monodromy on $H_1(N_{\chi}, \Q)$, 
respectively, on $\bigoplus_{i\le 2} H_i(N_{\chi}, \Q)$. 

Keeping the trivial monodromy assumption on $H_1(N_{\chi}, \Q)$, 
we show in Theorem \ref{prop:bns ak}  that 
$\Sigma^1(N_{\chi})\subseteq U^{\compl}$, where $U$ 
is the pullback to $N_{\chi}$ of a subspace arrangement 
in $H^1(G_{\G}, \R)$, depending explicitly on the graph $\G$. 
This enables us to construct in Example \ref{ex:rp2 again} 
finitely generated, yet {\em not}\/ finitely presentable 
Artin kernels $N_{\chi}$ for which the BNS invariant 
$\Sigma^1(N_{\chi})$ can be computed exactly. 

\subsection{K\"{a}hler and quasi-K\"{a}hler  manifolds}
\label{intro:kahler}

We conclude in \S\ref{sec:kahler} with a discussion of the 
$\Sigma$-invariants of K\"{a}hler and quasi-K\"{a}hler manifolds, 
and of their fundamental groups. 

The existence of a K\"{a}hler metric on a compact 
complex manifold $M$ puts strong constraints on 
its topology.  For instance, $M$ must be a formal 
space \cite{DGMS}, and thus, $G=\pi_1(M)$ must 
be $1$-formal. Hence, $\Sigma^1(G)$ lies in the 
complement of $\RR^1_1(M, \R)$, which, by \cite{DPS-jump}, 
is a finite union of rationally defined linear subspaces. 
Now suppose $b_1(M)>0$ (otherwise, of course, $\Sigma^1(G)=\emptyset$).
Using a recent result of Delzant \cite{De1}, we show in 
Theorem \ref{thm:kahler} that $\Sigma^1(G)=\RR^1_1(M, \R)^{\compl}$ 
precisely when $M$ admits no elliptic pencils with multiple fibers. 

Similar, yet weaker results hold when $X$ is a 
quasi-K\"{a}hler manifold (i.e., $X$ is the complement 
of a normal-crossing divisor in a compact K\"{a}hler manifold), 
and $G=\pi_1(X)$. The inclusion 
$\Sigma^1(G)\subseteq \RR^1_1(X, \R)^{\compl}$ 
fails in this generality. Nevertheless, we always have the upper bound 
\begin{equation}
\label{eq:intro qk}
\Sigma^1(G)\subseteq TC_1^{\R}(\VV^1_1(X, \C))^{\compl} ,
\end{equation}
for a quasi-K\"{a}hler manifold $X$. We establish \eqref{eq:intro qk} 
by using a deep result of Arapura \cite{Ar}, which implies that all 
irreducible components of $\VV^1_1(X, \C)$ containing $1$
are subtori.

When $X$ is the complement of a complex hyperplane arrangement, 
another basic result, due to Esnault, Schechtman, and Viehweg \cite{ESV},
implies that the exponential formula holds for $X$, in arbitrary degree 
and depth. Consequently, we obtain the combinatorial upper bound 
\begin{equation}
\label{eq=arrbound}
\Sigma^q(X, \Z)\subseteq \Big( \bigcup_{i\le q} 
\RR_1^i(X, \R)\Big)^{\compl}, \quad  \text{for all $q\ge 0$} .
\end{equation}

\section{Homological finiteness properties} 
\label{sec:finite}

We start by recalling some standard finiteness notions, 
for both spaces and groups. 

\subsection{Finiteness length}
\label{subsec:finlen}

Throughout, $X$ will be a connected CW-complex, 
with skeleta $\{X^{(k)}\}_{k\ge 0}$.    
Without loss of generality, we will assume $X$ has 
a single $0$-cell, call it $x_0$.  Furthermore, we will 
assume $X$ has finitely many $1$-cells. Set 
\begin{equation}
\label{eq:kappa}
\varkappa(X)=\sup\, \{ k \mid \text{$X^{(k)}$ is finite}\}.
\end{equation}
In the terminology of Bieri \cite{Bi99}, this is the 
{\em finiteness length}\/ of $X$. Note that 
$1\le \varkappa(X)\le \infty$, with the upper bound 
attained precisely when $X$ is of finite-type. 

\subsection{Homology with twisted coefficients}
\label{subsec:coeff}

Let $p\colon \wX \to X$ be the universal cover of our CW-complex.   
The cell structure on $X$ lifts to a cell structure on $\wX$.  
Fixing a lift $\tilde{x}_0\in p^{-1}(x_0)$ identifies 
the fundamental group $G=\pi_1(X,x_0)$ with the 
group of deck transformations of $\wX$, which 
permute the cells.  Therefore, we may view the 
cellular cell complex $C_{\bullet}(\wX, \Z)$, 
with differential $\tilde\partial$, as a chain complex 
of free left modules over the group ring $\Z{G}$.  

Given a right $\Z{G}$-module $A$, consider the chain complex 
$C_{\bullet}(X,A):=A\otimes_{\Z{G}} C_{ \bullet}(\wX, \Z)$, 
with differential $ \id_A\otimes \tilde\partial$.  
The homology groups of $X$ with coefficients 
in $A$ are then defined as $H_i(X,A):= H_i(C_{\bullet}(X,A))$. 

For the purpose of computing homology with certain twisted 
coefficients, in the finiteness range of $X$, the following 
lemma allows us to replace $X$ by a {\em finite}\/ CW-complex.

\begin{lemma}
\label{lem:skeleton}
Given a CW-complex $X$ as above, and an integer $1\le k\le \varkappa(X)$, 
there exists a finite CW-complex $Y$ of dimension at most $k+1$, 
with $Y^{(k)}=X^{(k)}$, and a map $f\colon Y\to X$ such that,  
for every commutative ring $S$, and every ring map    
$\Z{G}\to S$, the induced homomorphism, 
$f_*\colon H_i(Y,S)\to H_i(X,S)$, is an $S$-module 
isomorphism, for all $i\le k$. 
\end{lemma}

\begin{proof}
Let $X^{\ab}$ be the maximal abelian cover of $X$, corresponding 
to the abelianization map $\ab\colon G\surj G_{\ab}$.  The cellular 
chain complex $(C_i(X^{\ab}),\partial^{\ab}_i)_{i\ge 0}$ 
is a chain complex of free modules over the (commutative) 
Noetherian ring $R=\Z{G}_{\ab}$.  Since $C_k(X^{\ab})$ is 
finitely generated as an $R$-module, the $R$-submodule 
$B_k(X^{\ab})=\im (\partial^{\ab}_{k+1})$ is also finitely 
generated, let's say, by the images of $(k+1)$-cells 
$e_1,\dots, e_r$. 

Set $Y:=X^{(k)} \cup e_1 \cup \dots \cup e_r$, a finite subcomplex 
of $X^{(k+1)}$. Let $f\colon Y\to X$ be the inclusion map.  
By construction, $f_*\colon H_i(Y,R)\to H_i(X,R)$ is an 
$R$-isomorphism, for  all $i\le k$.  Since any 
ring map $\Z{G}\to S$ with $S$ commutative must 
factor through $R$, the conclusion follows. 
\end{proof}

\subsection{Finiteness properties of groups}
\label{subsec:ff fp}

A group $G$ is said to satisfy property $\FF_k$, for some $k \ge 1$, 
if $G$ admits a classifying space $K(G,1)$ with finite $k$-skeleton. 
If there is a finite-type (respectively, finite) $K(G,1)$, then $G$ 
is of type $\FF_{\infty}$ (respectively, $\FF$).  For example, 
finite groups are of type $\FF_{\infty}$, while fundamental 
groups of closed, aspherical manifolds are of type $\FF$. 
Note that the $\FF_1$ property is equivalent to $G$ being finitely 
generated, while the $\FF_2$ property is equivalent to $G$ being 
finitely presentable.  

These notions have an algebraic analogue. Let $A$ be a 
module over a ring $R$.  The $R$-module $A$ is said to be 
of type $\FP_k$ if there is a projective $R$-resolution  
$P_{\bullet}\to A$, with $P_i$ finitely generated for all $i\le k$. 
Modules of type $\FP_{\infty}$ and $\FP$ are similarly defined.

Now let $G$ be a group, and $R$ a commutative ring 
with $0\ne 1$.  Then $G$ is said 
to satisfy property $\FP_k(R)$ if the trivial $R{G}$-module $R$ 
is of type $\FP_k$.   The $\FP_{\infty}(R)$ and $\FP(R)$ 
properties are defined analogously. 
We will mostly deal with property $\FP_k=\FP_k(\Z)$.  
If $G$ is of type $\FP_k$, 
then of course all homology groups $H_i(G, \Z)$, for $i\le k$, 
are finitely generated.

Clearly, if $G$ is of type $\FF_k$, then $G$ is of type $\FP_k$. 
The converse also holds, provided $k=1$, or $G$ is finitely 
presentable, and $k<\infty$.  However, as shown by 
Bestvina and Brady \cite{BB}, $\FP_k \not\Rightarrow \FF_k$, 
for any $k\ge 2$. 

It is worth noting that the $\FP$ properties behave 
well with respect to group extensions. 

\begin{lemma}[\cite{Bi81}, Proposition~2.7]
\label{lem:bieri}
Let $1\to N\to G \to Q \to 1$ be an exact sequence of groups, 
and assume $N$ is of type $\FP_{\infty}$.  Then $G$ is of type 
$\FP_k$ if and only if $Q$ is.
\end{lemma} 

\section{Characteristic varieties and Alexander varieties}
\label{sec:cjl}

In this section, we discuss two types of homology jumping 
loci associated to a space, and establish a tight relation 
between the two. 

\subsection{Characteristic varieties}
\label{subsec:cvs}   
As before, let $X$ be a connected CW-complex with finite 
$k$-skeleton, $k\ge 1$, and $G=\pi_1(X)$. 
Fix a field $\k$, and let $\k^{\times}$ be its group of units. 
The group of $\k$-valued characters, $\Hom(G,\k^{\times})$, 
parametrizes rank~$1$ local systems on $X$.   Computing 
homology groups with coefficients in these local systems 
leads to a natural filtration of the character group by 
homology jumping loci. 

\begin{definition}
\label{def:cvs}
The {\em characteristic varieties} of $X$ (over $\k$) 
are the Zariski closed sets 
\begin{equation*}
\label{eq:cvs}
\VV^i_d(X,\k)=\{\rho \in \Hom(G,\k^{\times})
\mid \dim_\k H_i(X,\k_{\rho})\ge d\},
\end{equation*}
defined for all integers $i$ and $d$ with 
$0\le i \le \varkappa(X)$ and $d>0$. 
\end{definition}

Here $\k_{\rho}$ is the $1$-dimensional $\k$-vector 
space, viewed as a module over the group ring $\k{G}$ 
via $g \cdot a = \rho(g)a$, for $g\in G$ and $a\in \k$. 
In each degree $i$, the varieties 
$\VV^i_d=\VV^i_d(X,\k)$ define a descending filtration, 
$\VV^i_1 \supseteq \VV^i_2\supseteq \cdots$, of the algebraic 
variety $\Hom(G,\k^{\times})$. In degree $0$, it is easy to check that
$\VV^0_1(X, \k)= \{ 1\}$, and $\VV^0_d (X, \k)=\emptyset$, for $d>1$.

In order to compute the characteristic varieties in 
degrees $i\le k$, for some $k\le \varkappa(X)$, we may 
assume that $X$ is a finite CW-complex of dimension 
$k+1$.  Indeed, in view of Lemma \ref{lem:skeleton}, 
there is a CW-complex $Y$ with these properties, and 
a map $f\colon Y\to X$ such that the induced morphism 
between character groups, $f^*\colon \Hom(\pi_1(X),\k^{\times}) 
\to \Hom(\pi_1(Y),\k^{\times})$, is an isomorphism, 
identifying $\VV^i_d(Y, \k)$ with $\VV^i_d(X, \k)$, for 
all $i\le k$ and all $d>0$. 

It is worth noting that the characteristic varieties only 
depend on the characteristic of the field of definition.  
More precisely, we have the following. 

\begin{lemma}
\label{lem:char} 
Let $\k\subseteq \mathbb{K}$ be a field extension, and let 
$\Hom(G,\k^{\times})\subseteq  \Hom(G,\mathbb{K}^{\times})$ 
be the corresponding inclusion of character groups. Then
\begin{equation*}
\label{eq:vkk}
\VV^i_d(X,\k)=\VV^i_d(X,\mathbb{K}) \cap \Hom(G,\k^{\times}).
\end{equation*}
\end{lemma}

\begin{proof}
The characteristic varieties are determinantal varieties of 
matrices defined over $\Z$. The claim follows. 
\end{proof}

If $G$ is a group of type $\FF_k$ ($k\ge 1$),   
we may define the characteristic varieties of $G$ as 
$\VV^i_d(G,\k):=\VV^i_d(K(G,1),\k)$, for $0\le i \le k$ and 
$d>0$.  

\begin{lemma}
\label{lem:kg1}
Let $X$ be a connected CW-complex with finite 
$k$-skeleton ($k\ge 1$), and $\widetilde{X}$ its 
universal cover.  If $\pi_i(\widetilde{X})=0$, 
for all $i\le k$, then:
\begin{romenum}
\item \label{c1} 
$G=\pi_1(X)$ is of type $\FF_k$.  
\item \label{c2}
$\VV^i_d(X,\k)=\VV^i_d(G,\k)$, for all $i\le k$, all $d>0$, 
and all fields $\k$. 
\end{romenum}
\end{lemma}

\begin{proof}
Using the assumption, we may construct a classifying space 
$K=K(G,1)$ with $K^{(k+1)}=X^{(k+1)}$.  The two claims 
follow at once.  
\end{proof}

In particular, for any connected CW-complex $X$ with finite 
$1$-skeleton, the degree $1$ characteristic varieties of $X$ 
depend only on its fundamental group: 
\begin{equation}
\VV^1_d(X,\k)=\VV^1_d(\pi_1(X),\k), \quad\text{for all $d>0$}. 
\end{equation} 

\subsection{Alexander varieties}
\label{subsec:alex}
Let $X$ be a space as above, and let $\nu\colon G\surj H$ be an 
epimorphism from $G=\pi_1(X,x_0)$  to an abelian group $H$ 
(which necessarily must be finitely generated). 
Denote by $X^{\nu}\to X$ the corresponding Galois 
cover of $X$, with group of deck transformations $H$.

Recall that every finitely generated module $A$ over the 
Noetherian ring $R=\k{H}$ has a finite presentation, 
say, $R^m\xrightarrow{M}  R^n \to A \to 0$. 
The $i$-th elementary ideal of $A$, denoted $E_i(A)$, 
is the ideal generated by the minors of size $n-i$ of 
the $n \times m$ matrix $M$, with the convention that 
$E_i(A)=R$ if  $i \ge n$, and $E_i(A)=0$ if $n-i>m$. 

For each $i \le \varkappa(X)$, the homology group 
$H_i(X^{\nu},\k)$ is a finitely generated $\k{H}$-module. 
The support loci of the elementary 
ideals of these modules define a filtration 
of the character group of $H$, which pulls back to $G$ via 
$\nu^*\colon \Hom(H,\k^{\times})\to \Hom(G,\k^{\times})$. 

\begin{definition}
\label{def:alex}
The {\em Alexander varieties} of $X$ (over $\k$) are 
the algebraic sets 
\begin{equation*}
\label{eq:supp loci}
\WW^i_d(X,\nu,\k)=\nu^*(V(E_{d-1}(H_i(X^{\nu},\k)))),
\end{equation*}
with $0\le i \le \varkappa(X)$ and $d>0$.
\end{definition}

In particular, $\WW^i_1(X,\nu,\k)=\nu^*(V(\ann H_i(X^{\nu},\k)))$
is the pull-back of the support of the $\k{H}$-module 
$H_i(X^{\nu},\k)$. 
Given a group $G$ of type $\FF_k$, 
one may define as before the Alexander varieties 
$\WW^i_d(G,\nu,\k)$ for $0\le i\le k$, and prove the 
analogue of Lemma \ref{lem:kg1}\eqref{c2}.  

Noteworthy is the case when $H=G_{\ab}$, and $\nu=\ab$; 
the corresponding modules, $H_i(X^{\ab},\k)$, are the 
classical {\em Alexander invariants} of $X$.  In particular, 
$H_1(X^{\ab},\k)=G'/G'' \otimes \k$.  
Another case of note is when $H$ is the 
maximal torsion-free abelian quotient of $G$, $G_{\abf}$, and 
$\nu=\abf$ is the canonical projection.   In this case, we 
simply write $\WW^i_d(X,\k):= \WW^i_d(X,\abf,\k)$. 

Now suppose $G$ admits a finite presentation, 
say, $G= \langle x_1, \dots, x_n \mid r_1, \dots, r_m \rangle$.  
The associated Alexander matrix, 
$A_G=\big( \partial r_i/\partial x_j  \big)^{\abf}\colon 
(\Z{G_{\abf}})^m \to (\Z{G_{\abf}})^n$, is obtained by abelianizing the 
matrix of Fox derivatives of the relators, see \cite{Fo}. 
The codimension $d$ minors of $A_G$ define the 
variety $\WW^1_d(G,\k)$, away from the trivial character $1$.  

\begin{example}
\label{ex:lyndon}
Consider the Laurent polynomial ring 
$\Z[\Z^n]=\Z[t_1^{\pm 1},\dots ,t_n^{\pm 1}]$, 
and pick elements $v_1, \dots ,v_n$ 
satisfying the equation $\sum_{j=1}^n (t_j -1) v_j=0$. 
Let $F_n=\langle x_1,\dots, x_n\rangle$ be the 
free group of rank $n$, with abelianization map 
$\ab\colon F_n\to \Z^n$, $x_j \mapsto t_j$.  
By a result of R.~Lyndon recorded in \cite{Fo}, 
there exists an element $r\in F_n'$ such that 
$v_j= (\partial r/\partial x_j)^{\ab}$, for all $j$. For $n=2$,
the commutator-relator group 
$G:=\langle x_1, x_2\mid r\rangle$ 
has first characteristic variety 
$\VV^1_1(G,\k)= \{ 1\} \cup V(\Delta)\subseteq (\k^{\times})^2$, 
where $\Delta=\gcd (v_1, v_2)$ is the 
Alexander polynomial of $G$; see for instance \cite{DPS-codone}.  
\end{example}

\subsection{Comparing the two sets of varieties} 
\label{subsec:compare}
We adopt the 
spectral sequence approach of Dimca and Maxim 
from \cite[\S 2.5]{DM}, in a broader context.  

Recall $X$ is a connected CW-complex with 
$\varkappa(X)\ge 1$, and $G$ is its fundamental group. 
As before, let $\nu\colon G\surj H$ be an epimorphism 
to an abelian group $H$. Consider the group ring $R=\k{H}$.   
A character $\rho\colon H\to \k^{\times}$ corresponds to a 
ring map $\bar\rho\colon R\to \k$, which in turn corresponds 
to the maximal ideal $\m=\ker \bar\rho$ in $R$. Denote 
by $M_{\m}$ the localization of an $R$-module 
$M$ at $\m$.  Clearly,  $\bar\rho$ factors through 
a ring morphism $\bar\rho_{\m}\colon R_{\m}\to \k$.

Let $\nu^*\colon \Hom(H,\k^{\times}) \to \Hom(G,\k^{\times})$ 
be the induced homomorphism. 
The key to comparing the characteristic and Alexander 
varieties is the K\"{u}nneth spectral sequence associated  
to the chain complex $C_*(X,R_{\m})$ and the change-of-rings 
map $\bar\rho_{\m}$: 
\begin{equation}
\label{eq:ss local}
E^2_{s,t}=\Tor^{R_{\m}}_s (H_t(X^{\nu},\k)_{\m}, \k_{\bar{\rho}_{\m}})
\Rightarrow H_{s+t}(X,\k_{\nu^*\rho}).
\end{equation}

\begin{theorem}
\label{thm:vw}
For all $0\le k\le \varkappa(X)$, 
\begin{equation}
\label{eq:vw}
\bigcup_{i=0}^{k} \WW^i_{1}(X,\nu,\k)= 
\im(\nu^*)\cap 
\Big(\bigcup_{i=0}^{k} \VV^i_{1}(X,\k)\Big).
\end{equation}
\end{theorem}

\begin{proof}
The proof is by induction on $k$. Clearly,  $H_0(X^{\nu},\k)=\k$ 
is isomorphic to $R/I$, where $I$ is the augmentation ideal; 
thus, $\WW^0_1(X,\nu,\k)=\supp R/I =V(I)=\set{1}$.  On the 
other  hand, $\VV^0_1(X,\k)=\set{1}$, and so the two sides 
of \eqref{eq:vw} agree for $k=0$. 

($\supseteq$) Suppose 
$\rho \not\in\bigcup_{i=0}^{k} \supp H_i(X^\nu,\k)$.  Then 
$E^2_{s,t}=0$, for $t\le k$, and so $E^{\infty}_{s,t}=0$, for $t\le k$. 
Consequently, $H_{i}(X,\k_{\nu^* \rho})=0$, for all $i\le k$, 
which means that 
$\rho \not\in\bigcup_{i=0}^{k} \VV^i_1(X,\k)$.

($\subseteq$)  Suppose 
$\rho \in\bigcup_{i=0}^{k} \supp H_{i}(X^\nu,\k)$. 
We may assume $\rho \not\in \supp H_{t}(X^\nu,\k)$, 
for all $t<k$, and $\rho \in \supp H_{k}(X^\nu,\k)$, for 
otherwise we'd be done, by the induction hypothesis. 
These assumptions mean that $E^2_{s,t}=0$, for $t< k$, 
and $H_{k}(X^\nu,\k)_{\m}\ne 0$. Thus:
\[
H_{k}(X,\k_{\nu^*\rho})=E^{\infty}_{0,k}=E^{2}_{0,k}=
H_{k}(X^{\nu},\k)_{\m}\otimes_{R_{\m}} \k_{\bar{\rho}_{\m}}=
\frac{H_{k}(X^{\nu},\k)_{\m}}{\m \cdot H_{k}(X^{\nu},\k)_{\m}},
\]
which is non-zero, by Nakayama's Lemma. Hence, 
$\nu^*(\rho)\in \VV^k_1(X,\k)$, and we're done. 
\end{proof}

We will denote the character group $\Hom (G, \k^{\times})$ 
by $\T_G (\k)$. The notation $\T_G^0(\k)$ will be used for
the subgroup $\abf^*(\T_{G_{\abf}}(\k))$, where $\abf$ is 
the projection onto the maximal free abelian quotient of $G$. 
When $\k =\C$, we will simply write $\T_G$, respectively 
$\T_G^0$. The following is an immediate corollary of 
Theorem \ref{thm:vw}.

\begin{corollary}
\label{cor:vw}
For all $0\le k\le \varkappa(X)$, 
\begin{equation*}
\label{eq:vw2}
\bigcup_{i=0}^{k} \WW^i_{1}(X,\ab, \k)= \bigcup_{i=0}^{k} \VV^i_{1}(X,\k),
\ \text{and} \quad 
\bigcup_{i=0}^{k} \WW^i_{1}(X, \k)= \T_G^0(\k) \cap 
\Big ( \bigcup_{i=0}^{k} \VV^i_{1}(X,\k)\Big) .
\end{equation*}
\end{corollary}

\section{Tangent cones and exponential tangent cones}
\label{sec:tcone}

In this section, we discuss two types of tangent cones 
associated to a Zariski closed subset of the character variety 
$\Hom (G, \C^{\times})$: the usual kind, defined by the 
initial ideal, and a less standard one, coming from the 
exponential map. We start with the latter.

\subsection{The exponential map}
\label{subsec:exp}

Let $G$ be a finitely generated group. 
Its character group, $\T_G=\Hom (G, \C^{\times})$, may 
be identified with the cohomology group $H^1(G, \C^{\times})$. 
Note that $\T_G$
is a complex algebraic group, with identity $1\in \T_G$ the 
constant character $G\to \C^{\times}$, $g\mapsto 1$, and
$\T_G^0$ is the connected component of $1\in \T_G$. 

Consider the complex vector space $\Hom (G, \C)=H^1(G, \C)$.  
Since $G$ is finitely generated, the dimension of this vector space, 
$n:=b_1(G)$, is finite. 
The exponential map $\C\to \C^{\times}$, $z\mapsto e^z$ 
is a group homomorphism.  As such, it defines a coefficient 
homomorphism, 
\begin{equation}
\label{eq:exp map}
\exp\colon H^1(G, \C) \to H^1(G, \C^{\times}), 
\end{equation}
whose image is $\T_G^0$.  Clearly, $\T_G^0=(\C^{\times})^n$, an 
algebraic torus of dimension $n$. 

\subsection{Exponential tangent cones}
\label{subsec:exp tcone}
We now recall a notion introduced in \cite{DPS-jump}. 
\begin{definition}
\label{def:exp tc}
Let $W$ be a Zariski closed subset of $\T_G$. 
The {\em exponential tangent cone} of $W$ at $1$ is 
the set
\begin{equation*}
\label{eq:tau1}
\tau_1(W)= \{ z\in H^1(G, \C) \mid \exp(tz)\in W,\ \text{for all $t\in \C$} \}.
\end{equation*}
\end{definition}

This construction enjoys several pleasant properties, some of 
which we record below.

\begin{lemma}
\label{lem:exp tc}
With notation as above, 
\begin{romenum}
\item \label{t1}
$\tau_1(W)$ is a homogeneous subvariety of $H^1(G, \C)$.
\item \label{t2}
$\tau_1(W)$ depends only on the analytic germ of $W$ at the identity.
\item \label{t2.5}
$1\notin W$ if and only if  $\tau_1(W)=\emptyset$. 
\item \label{t3}
If $W_1\subseteq W_2$, then $\tau_1(W_1)\subseteq \tau_1(W_2)$.
\item \label{t4}
$\tau_1$ commutes with finite unions. 
\item \label{t5}
$\tau_1$ commutes with arbitrary intersections. 
\end{romenum}
\end{lemma}

The proof is an exercise in the definitions. A deeper property 
of the exponential tangent cone is given in the following lemma. 

\begin{lemma}[\cite{DPS-jump}]
\label{lem:tau1 lin}
For any subvariety $W\subseteq \T_G$, the exponential tangent cone 
$\tau_1(W)$ is a finite union of rationally defined linear subspaces 
of $H^1(G, \C)$.
\end{lemma}

Let us describe these  subspaces concretely.  
Since $\tau_1$ commutes with intersections, 
it is enough to consider the case $W=V(f)$, where 
$f= \sum_{u\in S} c_u t_1^{u_1}\cdots t_n^{u_n}$ 
is a non-zero Laurent polynomial, with support $S\subseteq \Z^n$. 
We may assume $f(1)=0$, for otherwise, $\tau_1(W)=\emptyset$. 

Let $\mathcal{P}$ be the set of partitions 
$p=S_1 \coprod \dots \coprod S_r$ of $S$, 
having the property that $\sum_{u\in S_i} c_u =0$, 
for $i=1, \dots, r$. For each such partition, define 
\begin{equation}
\label{eq:lp}
L(p):= \{ z\in \C^n \mid \langle u-v, z\rangle =0, 
\; \forall u,v\in S_i, \; \forall 1\le i\le r \}  . 
\end{equation}
Clearly, $L(p)$ is a rational linear subspace in 
$\C^n$.  Then, as shown in \cite[Lemma~4.3]{DPS-jump}, 
\begin{equation}
\label{eq:tau1w}
\tau_1(W)= \bigcup_{p\in \mathcal{P}} L(p).
\end{equation}
Note that only maximal partitions contribute to this union. 
Indeed, if $p$ is a refinement of $q$, then $L(p)\supseteq L(q)$. 

\subsection{Tangent cones}
\label{subsec:tcone}

For a Zariski closed subset $W\subseteq \T_G$, let $TC_1(W)$ be 
the tangent cone at $1$.  The classical definition, due to Whitney, 
goes as follows. Let $I$ be the ideal in $\C \{ z_1,\dots , z_n\}$ defining 
the germ of $W$ at $1$, and let $\init (I)$ be the ideal spanned by 
the initial forms of non-zero elements of $I$. Then  $TC_1(W)$ is 
the affine cone in $H^1(G,\C)=\C^n$ given by the zero-set of $\init (I)$. 

It is readily checked that $TC_1$ satisfies properties analogous to 
those satisfied by $\tau_1$, as enumerated in Lemma \ref{lem:exp tc}\eqref{t1}-\eqref{t4}. 
We now compare the two definitions.

\begin{prop}
\label{prop:tctc}
Let $W$ be a subvariety of $\T_G$.  Then:
\begin{romenum}
\item \label{tc1} 
$\tau_1(W) \subseteq TC_1(W)$.
\item \label{tc2} 
If all irreducible components of $W$ containing $1$ are subtori, 
then $\tau_1(W)=TC_1(W)$.
\end{romenum}
\end{prop}

\begin{proof}
Let $n=b_1(G)$, and identify $H^1(G,\C)=\C^n$.

Part \eqref{tc1}. Let $g\in \C \{ z_1,\dots , z_n\}$ be a non-zero analytic 
function vanishing on the germ of $W$ at $1$, where $z_i=t_i -1$. Let 
$g_0=\init(g)$ be its lowest homogeneous term. Suppose $z\in \tau_1(W)$. 
Then $g(e^{tz_1}-1,\dots, e^{tz_n}-1)=0$, for all $t\in \C$ near $0$. 
Note that
\[
g(e^{tz_1}-1,\dots, e^{tz_n}-1)= g_0(tz_1 ,\dots,tz_n  )+
\text{higher order terms}.
\]
It follows that $g_0(z)=0$, and thus $z\in TC_1(W)$. 

Part \eqref{tc2}. Using the fact that both $\tau_1$ and $TC_1$ commute
with finite unions, we may assume $W$ is a single subtorus, 
passing through $1$.  In this case, both $\tau_1(W)$ 
and $TC_1(W)$ are equal to the tangent space $T_1(W)$, 
and we are done. 
\end{proof}

The inclusion $\tau_1(W) \subseteq TC_1(W)$ can 
be strict even when $1$ is a smooth point of $W$, as the 
next example shows.  

\begin{example}
\label{ex:ttc}
Consider the Laurent polynomial 
$f=t_1+t_2-2\in \Z[t_1^{\pm 1}, t_2^{\pm 1}]$. 
Applying the method described in Example \ref{ex:lyndon} 
to the polynomials $v_1=f(t_2-1)$ and $v_2=f(1-t_1)$, 
we may find a commutator-relator group 
$G=\langle x_1,x_2 \mid r\rangle$ with 
Alexander polynomial $f$.  We then have 
$\VV^1_1(G,\C)=\{(t_1,t_2)\in 
(\C^{\times})^2 \mid t_1+t_2=2\}$. 
Letting  $W=\VV^1_1(G,\C)$, 
an easy calculation gives $\tau_1(W)= \{z\mid e^{tz_1}+
e^{tz_2}=2\}=\{0\}$, whereas $TC_1(W)=\{z\mid z_1+z_2=0\}$. 
\end{example}

\section{Resonance varieties, formality, and the exponential formula}
\label{sec:res var}

In this section, we discuss the jumping loci 
associated to the cohomology ring of a space. 
Under a formality assumption, these loci are 
related to the characteristic varieties via the 
exponential map. 

\subsection{Resonance varieties}
\label{subsec:rv}

Let $X$ be a connected CW-complex with finitely many 
$k$-cells, for $1\le k \le \varkappa(X)$.  Fix a field $\k$, 
and let $A=H^* (X,\k)$ be the cohomology ring, with graded 
ranks the Betti numbers $b_i= \dim_{\k} A^i$.   
If $\ch\k =2$, assume additionally that $H_1(X,\Z)$ 
has no $2$-torsion. (This mild condition presents no 
problems for us---it will be verified in all the examples we  
deal with here.)   For each $a\in A^1$, we then have 
$a^2=0$. Thus, right-multiplication by $a$ defines 
a cochain complex 
\begin{equation}
\label{eq:aomoto}
\xymatrix{(A , \cdot a)\colon  \ 
A^0\ar^(.66){a}[r] & A^1\ar^{a}[r] & A^2  \ar[r]& \cdots},
\end{equation}
known as the {\em Aomoto complex}.  Let 
$\beta_i(A,a) = \dim_{\k} H^i(A,\cdot a)$ be the Betti 
numbers of this complex.   The jumping loci for the 
Aomoto-Betti numbers define a natural filtration 
of the affine space $A^1=H^1(X,\k)$. 

\begin{definition}
\label{def:rvs}
The {\em resonance varieties} of $X$ (over $\k$) are the 
algebraic sets 
\begin{equation*}
\label{eq:rvs}
\RR^i_d(X,\k)=\{a \in A^1 \mid 
\beta_i(A,a) \ge  d\}, 
\end{equation*}
with $0\le i \le \varkappa(X)$ and $d>0$. 
\end{definition}

It is readily seen that each of these sets is a 
homogeneous algebraic subvariety  of $A^1=\k^{b_1}$.  
Indeed, $\beta_i(A,x a) =\beta_i(A,a)$, for all $x\in \k^{\times}$, 
and homogeneity follows. In each degree $i$ as above, the 
resonance varieties provide a descending filtration,
\begin{equation}
\label{eq:res filt}
H^1(X,\k) \supseteq \RR^i_1(X,\k) \supseteq \cdots 
\supseteq \RR^i_{b_i}(X,\k) \supseteq \RR^i_{b_{i}+1}(X,\k)=\emptyset.
\end{equation} 
Note that, if $A^i=0$, then $\RR^i_d(X,\k)=\emptyset$, for all $d>0$. 
In degree $0$, it is straightforward to check that $\RR^0_1(X, \k)= \{ 0\}$,
and $\RR^0_d(X, \k)= \emptyset$, for $d>1$.

The resonance varieties depend only on the characteristic of the 
ground field for the cohomology of $X$.  More precisely, if 
$\k\subseteq \mathbb{K}$ is a field extension, then
\begin{equation}
\label{eq:res ext}
\RR^i_d(X,\k)=\RR^i_d(X,\mathbb{K}) \cap H^1(X,\k).
\end{equation}

If $G$ is a group of type $\FF_k$ (and $G_{\ab}$ has no $2$-torsion,
if $\ch\k=2$), we may define the resonance varieties of $G$ as 
$\RR^i_d(G,\k):=\RR^i_d(K(G,1),\k)$, for $i\le k$. It is 
readily checked that $\RR^1_d(X,\k)=\RR^1_d(\pi_1(X),\k)$. 
The analogue of Lemma \ref{lem:kg1} holds, with the characteristic 
varieties replaced by resonance varieties.

\begin{remark}
\label{rem:holo res}
Suppose $G$ is a finitely generated group (with torsion-free 
abelianization, if $\ch\k=2$).  Let $\nabla_G\colon 
\Hom_{\k}(H^2(G, \k), \k) \to H_1(G, \k)\wedge H_1(G, \k)$ 
be the comultiplication map, dual to the cup-product map 
on $H^1(G, \k)$. The  {\em holonomy Lie algebra}\/ 
of $G$, over the field $\k$, denoted $\h(G, \k)$, is the quotient 
of the free Lie algebra on the $\k$-vector space $H_1(G, \k)$ 
by the ideal generated by the image of $\nabla_G$. 
Closely related is the {\em infinitesimal Alexander invariant},
$\bb (G, \k):= \h'(G, \k)/ \h''(G, \k)$, viewed as a (finitely generated) 
module over the polynomial ring $S=\Sym(H_1(G, \k))$; 
see \cite[Theorem 6.2]{PS-chen} for more details.  

It turns out that the resonance variety $\RR^1_d(G,\k)$ coincides, 
away from the origin, with the zero set of the elementary ideal 
$E_{d-1}(\bb (G, \k))$.  This is proved in \cite[Lemma 4.2]{DPS-serre} 
when $G$ is finitely presented, but the argument 
works as well in this broader context. 
\end{remark}

\subsection{The exponential formula}
\label{subsec:exp formula}
We now specialize to the case $\k=\C$.  Let $G$ be the 
fundamental group of $X$.  Consider the exponential map,  
$\exp\colon \Hom(G,\C)\to \Hom(G,\C^{\times})$, given by 
$\exp(\rho)(g)=e^{\rho(g)}$.   

\begin{definition}
\label{def:exp}
Let $k\le \varkappa(X)$ be a positive integer.  
We say the {\em exponential formula}\/ holds for $X$ 
(in degree $i\le k$, and for depth $d>0$) if the exponential 
map restricts to an isomorphism of analytic germs,
\begin{equation}
\label{eq:exp}
\exp\colon (\RR^i_d(X,\C),0)  
\xrightarrow{\,\simeq\,} (\VV^i_d(X,\C),1).
\end{equation}
\end{definition}

Clearly, the exponential formula always holds in degree $i=0$, 
for arbitrary depth.  The geometric motivation behind this definition 
comes from the following proposition.

\begin{prop}
\label{prop:tc formula} 
If the exponential formula \eqref{eq:exp} holds for $X$, then 
the following {\em tangent cone formula}\/ also holds: 
\begin{equation}
\label{eq:tc formula}
TC_1(\VV_d^i(X, \C))=\RR_d^i(X, \C).
\end{equation}
Moreover, $\tau_1 (\VV_d^i(X, \C))= TC_1(\VV_d^i(X, \C))$, 
and so $\RR_d^i(X, \C)$ is a finite union of linear subspaces 
defined over $\Q$.
\end{prop}

\begin{proof}
Let $W=\VV_d^i(X, \C)$.  By Lemma \ref{lem:exp tc}, 
$\tau_1(W)$  depends only on the analytic germ of $W$ at $1$. 
From the assumption, $\tau_1(W)=\RR_d^i(X, \C)$. 
But $\tau_1(W)\subseteq TC_1(W)$, by Proposition 
\ref{prop:tctc}\eqref{tc1}.  The first two statements follow 
at once.   The last statement now follows from 
Lemma \ref{lem:tau1 lin}.  
\end{proof}

\begin{remark}
\label{rem:lib}
For an arbitrary connected CW-complex with finite $1$-skeleton, 
\begin{equation}
\label{eq:lib}
TC_1(\VV_d^i(X, \C))\subseteq \RR_d^i(X, \C),
\end{equation}
for all $i< \varkappa(X)$ and $d>0$.  This follows from a 
result of Libgober \cite{Li} 
and Lemma \ref{lem:skeleton}. 
In general, though, the inclusion is strict; see \cite{MS} 
and \cite{DPS-serre}. 
\end{remark}

\subsection{Formality properties}
\label{subsec:formal}
Let $G$ be a  finitely generated group. Following Sullivan \cite{Su77}, 
$G$ is said to be {\em $1$-formal}\/ if its Malcev completion (in the 
sense of Quillen \cite{Q}) is a quadratically presented, complete Lie 
algebra; we refer to \cite{PS-chen} for more details.  Examples of 
$1$-formal groups are finitely generated Artin groups \cite{KM},  
and finitely presented Bestvina-Brady groups \cite{PS-bb}.  If 
$b_1(G)=0$ or $1$, then $G$ is $1$-formal.

A finite-type CW-complex $X$ is said to be {\em formal}\/ 
if its minimal model, as defined in \cite{Su77}, is 
quasi-isomorphic to $(H^*(X,\Q),0)$.  Examples of formal 
spaces are compact K\"{a}hler manifolds \cite{DGMS}, 
rational cohomology tori, and complements of complex 
hyperplane arrangements.  On the other hand, if $X$ is 
a nilmanifold, not homeomorphic to a torus, then $X$ 
is not formal. As noted by Sullivan, if $X$ is a formal space, 
then $\pi_1(X)$ is a $1$-formal group.  

The next result provides the main link between the $1$-formality 
property of a group and its (co)homology jumping loci.

\begin{theorem}[\cite{DPS-jump}]
\label{thm:dpsA}
Suppose $G$ is a $1$-formal group.  Then the 
exponential formula holds for $G$ in degree $1$, 
for all depths $d\ge 1$:
\begin{equation*}
\label{eq:exp1}
\exp\colon (\RR^1_d(G,\C),0)  
\xrightarrow{\,\simeq\,} (\VV^1_d(G,\C),1). 
\end{equation*}
\end{theorem}
In particular, $\tau_1 (\VV^1_d(G, \C))= 
TC_1(\VV^1_d(G, \C))=\RR^1_d(G, \C)$, 
a union of linear subspaces defined over $\Q$. 

\section{Homology of free abelian covers}
\label{sect:abel covers}

In this section, we recall a beautiful result of 
Dwyer and Fried \cite{DF}, which describes precisely 
how the Alexander varieties of a space control 
the homological finiteness properties of its free 
abelian covers.  We then offer some refinements, 
and draw some consequences.  

\subsection{A criterion for homological finiteness}
\label{subsec:dwyer-freed}
Let $X$ be a connected CW-complex (with single 
$0$-cell $x_0$), and $G=\pi_1(X,x_0)$.  Fix a 
coefficient field $\k$, which we shall assume for 
the rest of this section to be algebraically closed,  
and write $H_{\le k}(X,\k):=\bigoplus_{i=0}^{k} H_i(X,\k)$. 

Consider an epimorphism $\nu\colon G\surj \Z^r$, and let 
$X^{\nu}$ be the corresponding $\Z^r$-cover.  Denote by 
$\nu^*\colon \Hom(\Z^r,\k^{\times} ) \to \Hom(G,\k^{\times} )$ 
the induced homomorphism between character groups. 

\begin{theorem}[Dwyer and Fried \cite{DF}]
\label{thm:df}
Suppose $X$ has finitely many cells.  Then, for all $k\ge 0$, 
\[
\dim_{\k} H_{\le k} (X^{\nu}, \k) <\infty \same 
\text{$\im(\nu^*) \cap \Big(\bigcup_{i=0}^{k} \WW^i_1(X,\k)\Big)$ is finite}.
\]
\end{theorem}

Making use of the machinery developed in \S\ref{sec:cjl}, 
we may relax the hypothesis of Theorem \ref{thm:df}, 
from $X$ being finite to just $X^{(k)}$ being finite, and replace 
the Alexander varieties by the characteristic varieties.  
More precisely, we have the following corollary. 

\begin{corollary}
\label{cor:df cv}
Suppose $X$ has finitely many $1$-cells.  Then, 
for all $0\le k\le \varkappa(X)$, 
\[
\dim_{\k} H_{\le k} (X^{\nu}, \k) <\infty \same 
\text{$\im(\nu^*) \cap \Big( \bigcup_{i=0}^{k} 
\VV^i_1(X,\k)\Big)$ is finite}.
\]  
\end{corollary}

\begin{proof}
Corollary \ref{cor:vw} of Theorem \ref{thm:vw} yields
$\bigcup_{i=0}^{k} \WW^i_1(X,\k)= \T_G^0 (\k)\cap 
\big ( \bigcup_{i=0}^{k} \VV^i_1(X,\k)\big )$.  
By Lemma \ref{lem:skeleton}, we may replace $X$ by a 
finite CW-complex with the same $k$-skeleton, without 
changing $H_1(X,\Z)= G_{\ab}$, or the characteristic varieties in 
degrees $i\le k$.  The conclusion now follows from 
Theorem \ref{thm:df}. 
\end{proof}

\subsection{Homological finiteness and resonance}
\label{subsec:ac res}

Under favorable circumstances, the homological finiteness properties 
of a free abelian cover can be tested by means of the resonance 
varieties.  The next result proves Part \eqref{free2} of Theorem 
\ref{intro:thm3} from the Introduction. 

\begin{prop}
\label{prop:exp r}
Fix an integer $0\le k\le \varkappa(X)$, and suppose that, 
for each $i\le k$, 
\begin{alphenum}
\item \label{v1}
$\VV^i_1(X, \C) \cap \T_G^0$ is a union of subtori 
passing through $1$;
\item \label{v2}
$TC_1(\VV^i_1(X,\C))=\RR^i_1(X,\C)$. 
\end{alphenum}
Let $\nu\colon G\surj \Z^r$ be an epimorphism, and 
$\nu^*\colon H^1(\Z^r,\C) \to H^1(G,\C)$ the induced 
homomorphism. Then $H_{\le k} (X^{\nu}, \C)$ is finite 
dimensional if and only if
\[
\im(\nu^*) \cap  \Big ( \bigcup_{i\le k} \RR^i_1(X,\C)\Big ) =\set{0}.
\]
\end{prop}

\begin{proof}
Write $\big(\bigcup_{i=0}^{k} \VV^i_1(X,\C)\big) \cap \T_G^0 = 
\bigcup_{\alpha} T_{\alpha}$, with $T_{\alpha}$ subtori 
passing through $1$. Then:
\[
\begin{aligned}
\dim_{\C} H_{\le k} (X^{\nu}, \C) <\infty & \underset{\text{(i)}}{\same}
\nu^*(\Hom(\Z^r,\C^{\times})) \cap 
\big(\bigcup\nolimits_{\alpha} T_{\alpha}\big) \text{ is finite}\\
& \underset{\text{(ii)}}{\same}  
\nu^*(H^1(\Z^r,\C)) \cap \big(\bigcup\nolimits_{\alpha}
TC_1(T_{\alpha} )\big)=\set{0}\\
& \same
\nu^*(H^1(\Z^r,\C))  \cap \big(\bigcup\nolimits_{i\le k}
TC_1(\VV^i_1(X,\C)\big)=\set{0}\\
& \underset{\text{(iii)}}{\same}
\nu^*(H^1(\Z^r,\C))  \cap \big(\bigcup\nolimits_{i\le k}
\RR^i_1(X,\C)\big)=\set{0},
\end{aligned}
\]
where in (i) we used Corollary \ref{cor:df cv}, 
in (ii) we used $\ch(\C)=0$ and the hypothesis 
that $1\in \bigcap_{\alpha} T_{\alpha}$, and in 
(iii) we used hypothesis \eqref{v2}. 
\end{proof}

The next example shows that it is really necessary to insist in 
hypothesis \eqref{v1} that all subtori pass through the origin. 

\begin{example}
\label{ex:translated}
Consider the group $G=\langle x_1,x_2 \mid 
x_1^2x_2x_1^{-2}x_2^{-1}\rangle$, 
and let $X$ be the corresponding presentation $2$-complex. 
Clearly, $H_1(X,\Z)=\Z^2$, and the Alexander polynomial 
is $\Delta=t_1+1\in \Z[t_1^{\pm 1}, t_2^{\pm 1}]$. Thus, 
$\VV^1_1(X,\C)$ consists of the origin $1$, together with 
the translated subtorus  $\{(t_1,t_2)\in (\C^{\times})^2\mid t_1=-1\}$. 

Note that $\RR^1_1(X,\C)=\{0\}$, and so the tangent 
cone formula holds up to degree $1$, that is, 
$TC_1(\VV^i_1(X,\C))=\RR^i_1(X,\C)$, for $i=0,1$. 
Now let $\nu=\ab\colon G\surj \Z^2$. Then, 
$\dim_{\C} H_1(X^{\nu},\C)=\infty$, by Corollary \ref{cor:df cv}, 
although of course $\im(\nu^*)  \cap \RR^1_1(X,\C) =\{0\}$.  
\end{example}

\subsection{Infinite cyclic covers}
\label{subsec:zcov}

We now analyze in more detail the case $r=1$.  As before, 
let $G=\pi_1(X)$. Given a homomorphism $\nu\colon G\to \Z$, 
the induced homomorphism in homology, 
$\nu_{*}\colon H_1(X,\k)\to H_1(\Z,\k)=\k$, 
defines a cohomology class, $\nu_{\k}\in H^1(X,\k)$. 
As noted in  \cite[\S4.4]{PS-toric}, we have $\nu_{\k}^2 =0$, 
even if $\ch\k=2$.

The next result proves Part \eqref{free1} of Theorem 
\ref{intro:thm3} from the Introduction.

\begin{theorem}
\label{thm:nuc}
Let $\nu\colon G\surj \Z$ be an epimorphism, and let 
$X^{\nu}\to X$ be the corresponding infinite cyclic cover.  
Then, for all $k\le \varkappa(X)$, 
\[
\dim_{\C} H_{\le k} (X^{\nu}, \C) <\infty \same 
\nu_{\C}\not\in  \bigcup_{i=0}^{k} \tau_1\big(\VV^i_1(X,\C)\big).
\]
\end{theorem}

\begin{proof} 
Let $\nu^*\colon \T_{\Z}\to \T_{G}$ be the induced 
homomorphism on character tori.  It is readily seen 
that the image of $\nu^*$ is the curve 
$C=\{\exp(t \nu_{\C})\mid t\in \C\}\subseteq \T_G$. 

By Corollary \ref{cor:df cv}, $H_{\le k} (X^{\nu}, \C)$ is 
finite-dimensional if and only if $C$ intersects the variety 
$W=: \bigcup_{i=0}^{k} \VV^i_1(X,\C)$ in finitely many 
points. Since $C$ is irreducible, this happens precisely 
when $C\not\subseteq W$, which, by definition, is equivalent 
to $\nu_{\C}\notin \tau_1(W)$.   Applying 
Lemma \ref{lem:exp tc}\eqref{t4} ends the proof.  
\end{proof}

\begin{corollary}
\label{cor:exp 1}
Suppose the exponential formula \eqref{eq:exp} 
holds for $X$, up to degree $k$, and for depth $d=1$. Then,
\[
\dim_{\C} H_{\le k} (X^{\nu}, \C) <\infty \same 
\nu_{\C}\not\in  \bigcup_{i=0}^{k} \RR^i_1(X,\C).
\]
\end{corollary}
\begin{proof}
Combine Theorem \ref{thm:nuc} and Proposition \ref{prop:tc formula}. 
\end{proof}

Using Theorem \ref{thm:dpsA} and Corollary \ref{cor:exp 1}, 
we obtain the following immediate consequence. 

\begin{corollary}
\label{cor:b1 kernu}
Let $\nu\colon G\surj \Z$ be an epimorphism.  
If $G$ is $1$-formal,  then 
\[
b_1 (\ker {\nu}) <\infty \same 
\nu_\C \not\in  \RR^1_1(G,\C).
\]
\end{corollary}

\subsection{Homological finiteness as an open condition}
\label{subsec:openess}
Some of the above results can be interpreted qualitatively, 
as follows.  Let $X$ be a CW-complex as before, and let 
$G=\pi_1(X)$.  We may view an epimorphism 
$\nu\colon G\surj \Z^r$ as an element of 
$\Grass_r(H^1(X,\Q))$, the Grassmanian of $r$ 
planes in the vector space $H^1(X,\Q)$. 

Let $X^{\nu}\to X$ be the cover corresponding to $\nu$. 
As noted by Dwyer and Freed \cite{DF}, if $X$ is a finite 
CW-complex, the homological finiteness condition 
$\dim_{\C} H_{*} (X^{\nu}, \C) <\infty$ is an open 
condition for $r=1$.  In the next Proposition, we 
recover their result, in a slightly more general and 
more precise form.  

\begin{prop}
\label{prop:open}
Let $X$ be a connected CW-complex with finite $1$-skeleton. 
Then, for every $k\le \varkappa(X)$, the subset of 
$\Grass_1(H^1(X,\Q))$ defined by the condition
$\dim_{\C} H_{\le k} (X^{\nu}, \C) <\infty$ is the 
complement of a finite union of projective subspaces. 
\end{prop}

\begin{proof}
By Lemma \ref{lem:tau1 lin}, the set
$H^1(X, \Q) \cap \tau_1 \big(\bigcup_{i=0}^{k} \VV^i_1(X,\C)\big)$ 
is a finite union of rational linear subspaces.  
In particular, its projectivized complement  is a Zariski 
open subset in $\Grass_1(H^1(X,\Q))$. The claim now 
follows from Theorem \ref{thm:nuc}. 
\end{proof}

\begin{remark}
\label{rem:non-open}
For $r>1$, the above homological finiteness condition 
is no longer an open condition, as an example from \cite{DF} 
shows.  In particular, one cannot  strengthen 
Corollary \ref{cor:df cv} to the analogue of 
Theorem \ref{thm:nuc}, in the case when $r>1$.
\end{remark} 

\section{Bieri--Neumann--Strebel--Renz invariants}
\label{sec:bnsr}

In this section, we review the definition of the $\Sigma$-invariants 
of a group $G$, following Bieri--Neumann--Strebel \cite{BNS} and 
Bieri--Renz \cite{BR}, and discuss some of their basic properties. 

\subsection{$\Sigma$-invariants}
\label{subsec:bns}

Let $G$ be a finitely generated group. 
Pick a finite generating set for $G$, and let $\mathcal{C}(G)$ 
be the corresponding Cayley graph.  Given an additive real 
character $\chi\colon G\to \R$, let $\mathcal{C}_{\chi}(G)$ 
be the full subgraph on vertex set 
$G_{\chi}=\set{g \in G \mid \chi(g)\ge 0}$.  
In \cite{BNS}, Bieri, Neumann, and Strebel 
defined a subset $\Sigma^1(G)$   
of $\Hom(G,\R)\setminus \{0\}$, nowadays 
called the {\em BNS-invariant} of $G$, as follows. 

\begin{definition}[\cite{BNS}]
\label{def:bns}
The set $\Sigma^1(G)$  consists of those non-zero homomorphisms 
$\chi\colon G\to \R$ for which the graph $\mathcal{C}_{\chi}(G)$ 
is connected. 
\end{definition}

Clearly, $\Sigma^1(G)$ is a conical subset of the vector space 
$\Hom(G,\R)=H^1(G,\R)$.  It turns out that the BNS invariant 
$\Sigma^1(G)$ is an {\em open} subset of $\Hom(G,\R)$, and  
that this subset does not depend on the choice of generating 
set for $G$, see \cite{BNS}.

In \cite{BR}, Bieri and Renz generalized the BNS invariant, 
by defining a sequence of ``homological" invariants, 
$\{\Sigma^r(G,R)\}_{r\ge 0}$, for any commutative ring 
$R$ with $1\ne 0$ (viewed as a trivial $R{G}$-module).    

\begin{definition}[\cite{BR}]
\label{def:br}
The set $\Sigma^r(G,R)$ consists of those nontrivial homomorphisms 
$\chi\in \Hom (G, \R)$ for which $R$ is of type $\FP_r$, when viewed 
as a  module over the monoid ring $R{G_{\chi}}$.
\end{definition}

The BNSR invariants $\Sigma^r(G,R)$ form a descending 
chain of open subsets of $\Sigma^0(G,R)=\Hom(G,\R)\setminus\{0\}$. 
Clearly, if $b_1(G)=0$, then $\Hom(G,\R)=\{0\}$, and so 
$\Sigma^r(G,R)=\emptyset$, for all $r\ge 0$. 

As noted in \cite{BR}, if $\Sigma^r(G,\Z)\ne \emptyset$, then 
$G$ is of type $\FP_r$.  We will be especially interested 
in the invariants $\Sigma^r(G,\Z)$ and $\Sigma^r(G,\k)$, 
with $\k$ a field.  There is always an inclusion $\Sigma^r(G,\Z)
\subseteq \Sigma^r(G,\k)$, but this inclusion may be strict. 

One may also define ``homotopical" invariants, $\Sigma^r(G)$. 
If $G$ is of type $\FF_k$, the two kinds of geometric invariants 
are related by Hurewicz-type formulas: $\Sigma^1(G)= \Sigma^1(G,R)$, 
for all $R$, and $\Sigma^r(G)= \Sigma^2(G)\cap \Sigma^r(G,\Z)$, 
for $2\le r\le k$.

\subsection{Symmetry}
\label{subsec:symm}
If $M^3$ is a closed $3$-manifold, and 
$G=\pi_1(M)$, then $\Sigma^1(G) = -\Sigma^1(G)$, 
as proved in \cite[Corollary F]{BNS}.  In general, though, the 
$\Sigma$-invariants are not symmetric about the origin, 
as the next, well-known example shows. 

\begin{example}
\label{ex:1rel}
Consider the Baumslag-Solitar group 
$G=\langle x_1,x_2\mid x_1^{\,}x_2x_1^{-1}=x_2^2\rangle$. 
Clearly, $G_{\ab}=\Z$ is freely generated by $x_1$. Upon this 
identification, Theorem 7.3 from \cite{BR} implies that 
$\Sigma^1(G)=\R_{<0}$.  In particular, 
$\Sigma^1(G) \ne -\Sigma^1(G)$. 
\end{example}

For future reference, let us note the following straightforward 
criterion for symmetry of the BNSR invariants.  Suppose 
there is an automorphism 
$\alpha\colon G\to G$ such that   
$\alpha_* \colon H_1(G, \R) \to H_1(G, \R)$ equals 
$-\id$.  Then $\Sigma^r(G,R)= -\Sigma^r(G,R)$, 
for all $r\ge 1$.  

\subsection{Finiteness properties of kernels}
\label{subsec:bns fin}
Much of the importance of the $\Sigma$-invariants 
rests with the fact that they control the finiteness 
properties of kernels of projections to abelian quotients.  
Given a normal subgroup $N\triangleleft G$, write 
$S(G,N)= \{\chi \in \Hom (G,\R) \setminus \{0\} \mid \chi(N)=0\}$.  

\begin{theorem}[\cite{BNS}, \cite{BR}]
\label{thm:bnsr finp}
Let $G$ be a finitely generated group, and let 
$N\triangleleft G$ be a normal subgroup with $G/N$ 
abelian. Then 
\begin{romenum}
\item \label{bf1}
$N$ is of type $\FF_k$ if and only if 
$S(G,N) \subseteq \Sigma^k(G)$. 
\item \label{bf2}
$N$ is of type $\FP_k$ if and only if 
$S(G,N) \subseteq \Sigma^k(G, \Z)$. 
\end{romenum}
\end{theorem}

In particular, if $\chi\colon G\surj \Z$ is an epimorphism, 
then $N=\ker(\chi)$ is finitely generated if and only if 
$\{\pm \chi\}\subseteq \Sigma^1(G)$.  In view of the theorem 
of Stallings \cite{St1} mentioned in \S\ref{intro:alex thurston}, 
this implies the following:  A closed $3$-manifold $M$ 
fibers over $S^1$ if and only if $\Sigma^1(\pi_1(M))\ne \emptyset$.

\section{Novikov homology}
\label{sec:novikov}

In this section, we recall how the $\Sigma$-invariants of a  group $G$ 
may be reinterpreted in terms of the vanishing of certain $\Tor$ groups.  
In turn, this interpretation opens the way to extending the definition 
of these invariants to spaces. 

We start by reviewing a notion introduced by S.~P. Novikov 
in the early 1980s, and later generalized by J.-Cl. Sikorav \cite{Si}.  
We refer to M.~Farber's book \cite{Fa} for a comprehensive 
treatment, and to R.~Bieri \cite{Bi07} for further details. 

\subsection{The Novikov-Sikorav completion} 
\label{subsec:nov}

Let $G$ be a finitely generated group, and let $\chi\colon G\to \R$ 
be an additive character.  The {\em Novikov-Sikorav completion}\/ 
of the group ring $\Z{G}$ with respect to $\chi$, denoted 
$\widehat{\Z{G}}_{\chi}$, consists of all formal sums 
$\lambda =\sum_i n_i g_i$, with $n_i\in \Z$ and $g_i\in G$, 
having the property that the set of indices $ i$ for which 
$n_i \ne 0 $ and $\chi(g_i) \ge c$ is finite, for each $c\in \R$. 
With componentwise addition, and multiplication defined by 
$(\sum_i n_i g_i) \cdot  (\sum_j m_j h_j) = \sum_{i,j} (n_i m_j) g_i h_j$, 
the Novikov-Sikorav completion becomes a ring.  
Clearly, $\widehat{\Z{G}}_{\chi}$ contains $\Z{G}$ as a subring, 
and thus acquires a natural $G$-module structure. For more 
details, see \cite[p.~52]{Fa}. 

The above construction is a ring completion, in the following sense.
For each integer $s$, let $F_s$ be the abelian subgroup of $\Z{G}$ 
generated by those $g\in G$ with $\chi(g)\ge s$.  Requiring 
the decreasing filtration $\set{F_s}_{s\in \Z}$ to form a 
basis of open neighborhoods of $0$ defines a topology 
on $\Z{G}$, compatible with the ring structure.  It turns out that 
\begin{equation}
\label{eq:completion}
\widehat{\Z{G}}_{-\chi}=\varprojlim_{s}\, \Z{G}/F_s,
\end{equation}
the completion of $\Z{G}$ with respect to this topology. 
Moreover, the inclusion map between the respective 
topological rings, 
$\kappa_{\chi}\colon \Z{G} \hookrightarrow \widehat{\Z{G}}_{-\chi}$,
is the structural map for this completion.

The Novikov-Sikorav completion enjoys the following functoriality 
property.  Let $\phi\colon G\to H$ be a homomorphism, and 
$\bar\phi\colon \Z^{G}\to \Z^{H}$ its linear extension to 
formal sums. If $\chi\colon H\to \R$ is a character, then 
$\bar\phi$ restricts to a (topological) ring morphism, $\hat\phi \colon  
\widehat{\Z{G}}_{\chi\circ \phi} \to  \widehat{\Z{H}}_{\chi}$, 
between the corresponding completions. 

\subsection{Generalized $\Sigma$-invariants}
\label{subsec:defspaces}

In his thesis \cite{Si}, J.-Cl.~Sikorav reinterpreted the 
$\Sigma$-invariants of a finitely generated group 
$G$ in terms of the vanishing of its homology groups, 
with coefficients in $\widehat{\Z{G}}_{\chi}$.  Let us record 
here this basic result, in a form due to Bieri (Theorem 2 
from \cite{Bi07}, proved in an appendix by P.~Schweitzer).  

\begin{theorem}[\cite{Bi07}]
\label{thm:bns-novikov}
Let $G$ be a group of type $\FP_k$ ($k\ge 1$). 
For a character $\chi \in \Hom(G,\R)\setminus \{0\}$, 
and an integer $q\le k$,
\[
\chi \in \Sigma^q(G,\Z) \same 
\Tor ^{\Z{G}}_{i} (\widehat{\Z{G}}_{-\chi},\Z)=0, \ \text{for all $i\le q$}. 
\]
\end{theorem}

This result naturally leads to the following definition.

\begin{definition}
\label{def=bnsrx}
Let $X$ be a connected CW-complex with finite $1$-skeleton, 
and fundamental group $G=\pi_1(X)$. For each $q\ge 0$, set
\[
\Sigma^q(X, \Z):= \{ \chi\in \Hom(G, \R)\setminus \set{0} \mid
H_{\le q}(X, \widehat{\Z{G}}_{-\chi})=0 \}.
\]
\end{definition}

Clearly, if $G$ is a group of type $\FP_k$, then 
$\Sigma^q(G, \Z)= \Sigma^q(K(G, 1), \Z)$, for all $q\le k$. 
Thus, this definition generalizes the usual definition of the 
BNSR invariants.  At the same time, Proposition 7 from 
\cite{FGS} implies that the above definition 
coincides with the one introduced by Farber, Geoghegan, 
and Sch\"{u}tz, whenever $X$ is a finite CW-complex.  

For the record, let us note that the analogue of  
Lemma \ref{lem:kg1} holds for the $\Sigma$-invariants 
of spaces and groups, in the form given below, and 
with essentially the same proof. 

\begin{lemma}
\label{lem=gx}
Let $X$ be a connected CW-complex with finite 
$k$-skeleton ($k\ge 1$), with universal cover $\widetilde{X}$ 
and fundamental group $G$.  If $\pi_i(\widetilde{X})=0$, 
for all $i\le k$, then $\Sigma^i(X, \Z)= \Sigma^i(G, \Z)$, 
for all $i\le k$.
\end{lemma}

\subsection{Rational Novikov ring} 
\label{subsec:rat nov}

As before, let $G$ be a finitely generated group, and 
$\chi\colon G\to \R$  a non-zero homomorphism.  Let 
$\Gamma$ be the image of $\chi$.  The group $\Gamma$ 
is a lattice, say, $\Gamma \cong \Z^r$, with $r>0$. 
Moreover, $\chi$ factors as $\chi=\iota \circ \xi$, 
where $\xi\colon G\surj \Gamma$ is an epimorphism, 
and $\iota\colon \Gamma\inj \R$ is a monomorphism. 
By functoriality of Novikov-Sikorav completion, we have 
the commuting square on the left side of the diagram,
\begin{equation}
\label{eq:rat nov}
\xymatrixrowsep{7pt}
\xymatrixcolsep{7pt}
\entrymodifiers={=<1.75pc>}
\xymatrix{\Z{G}\ar^{\xi}[rr] \ar@{^{(}->}[dd] && 
\Z{\Gamma}\ar@{^{(}->}[dd] \ar@{^{(}->}[dr]\\
&&& \rat \ar@{^{(}->}[dl]\\
\widehat{\Z{G}}_{\chi} \ar^{\hat\xi}[rr] 
&& \widehat{\Z{\Gamma}}_{\iota}}
\end{equation}

The ring $\Z\Gamma$ is a Laurent polynomial ring, 
while its completion, $\widehat{\Z\Gamma}_{\iota}$,  
is a principal ideal domain (PID). There is a very 
useful intermediary ring, $\rat$, which Farber \cite{Fa} 
calls the {\em rational Novikov ring}\/ of $\G$ (with 
respect to $\iota$). The construction of $\rat$ goes 
as follows. 

A Laurent polynomial  
$p=\sum_{\gamma} n_{\gamma} \gamma\in 
\Z\G$ is called {\em $\iota$-monic}\/ if the greatest 
element in $\iota(\supp(p))$ is $0$, and $n_0=1$.    
By \cite[Lemma 1.9]{Fa}, every such polynomial
is invertible in  $\widehat{\Z\Gamma}_{\iota}$. 
The ring $\rat$ is then the localization of $\Z\G$ 
at the multiplicative subset $S$ of all $\iota$-monic 
polynomials.  As shown in \cite[Lemma 1.15]{Fa}, 
the ring $\rat=S^{-1} \Z\G$ is a PID. In fact, as 
shown by Pajitnov in \cite[Theorem 1.4]{Paj}, $\rat$ 
is a Euclidean ring.

\subsection{$\Sigma$-invariants and Novikov Betti numbers} 
\label{subsec:nov betti}
Now  let $X$ be a connected CW-complex, with finite 
$k$-skeleton ($k\ge 1$).  Let $G=\pi_1(X)$, and consider 
a character $\chi\colon G\xrightarrow{\xi} 
\Gamma \xrightarrow{\iota} \R$ factored 
as above. Taking advantage of the fact that the ring $\rat$ 
is both a $G$-module and a PID, one may define the 
{\em Novikov Betti numbers} of $X$ as
\begin{equation}
\label{eq:nov betti}
b_i(X,\chi):=\rank_{\rat} H_i(X,\rat), \quad 
\text{for $i\le k$}. 
\end{equation}

\begin{remark}
\label{rem:farber}
Using Lemma \ref{lem:skeleton}, we may assume in this  
definition that $X$ is a {\em finite}\/ CW-complex, which is 
the assumption that Farber makes in \cite{Fa}. 
\end{remark}

\begin{prop}
\label{thm:nov betti}
Let $X$ be a CW-complex with finite $1$-skeleton. Then, 
for all $q\le \varkappa (X)$, 
\begin{equation*}
-\chi \in \Sigma^q(X,\Z)\implies  H_i(X,\rat)=0,\ \text{for all $i\le q$}. 
\end{equation*}
In particular, $b_i(X,\chi)=0$, for $i\le q$. 
\end{prop}

\begin{proof}
Set $G=\pi_1(X)$. 
Let $-\chi \in \Sigma^q(X,\Z)$, and fix $i\le q$.   By 
Definition \ref{def=bnsrx}, the group $H_{i} (X, \widehat{\Z{G}}_{\chi})$ 
vanishes. Using the change-of-rings spectral sequence associated 
to the ring morphism $\hat\xi\colon \widehat{\Z{G}}_{\chi} \to 
\widehat{\Z{\G}}_{\iota}$, 
\begin{equation}
\label{eq:change2}
E^2_{st}= \Tor^{\widehat{\Z{G}}_{\chi}}_{s} (\widehat{\Z{\G}}_{\iota},
H_{t}(X, \widehat{\Z{G}}_{\chi}))
\Rightarrow H_{s+t}(X, \widehat{\Z{\G}}_{\iota}),
\end{equation}
we find that $H_i(X,\widehat{\Z{\G}}_{\iota})$ also vanishes.  
Using \cite[Proposition 1.29]{Fa}, we conclude that $H_i(X,\rat)=0$.
\end{proof}

\section{$\Sigma$-invariants and jumping loci over $\C$}
\label{sec:main}

In this section, we prove our main Theorem, relating the 
BNSR invariants of a space (or of a group) to its homology 
jump loci. We start with a result asserting that non-membership 
in the exponential tangent cones to the characteristic varieties 
is equivalent to vanishing of the respective Novikov-Betti numbers. 
This improves on a similar vanishing result, implied by Farber's 
Theorem 1.50 from \cite{Fa}.

\subsection{Tangent cones and Novikov-Betti numbers}
\label{subsec:nov-tau}
Let $X$ be a connected CW-complex with finite $k$-skeleton 
($k\ge 1$), and $G=\pi_1(X)$ its fundamental group.  
Identify $\Hom(G,\R)$ and $H^1(G,\R)$. 
For an algebraic subset $W\subseteq \T_G$, denote by  
$\tau_1^{\R}(W)=\tau_1(W)\cap H^1(G, \R)$ the set of 
real points on the exponential tangent cone of $W$ 
(and similarly for $TC_1$).  

\begin{theorem}
\label{thm:nov-tau}
Let $\chi\colon G\to \R$ be a non-zero character. 
Then, for each $0\le q\le k$, 
\[
\chi\not\in \tau^{\R}_1\Big(\bigcup_{i\le q} \VV_1^i(X, \C)\Big) 
\same b_i(X, \chi)=0, \ \text{for all $i\le q$}.
\]
\end{theorem}

\begin{proof} 
Factor $\chi=\iota\circ \xi$, with $\xi\colon G\surj \G$ 
and $\iota\colon \G\inj \R$ as before.  
Let $\mathcal{O}_{\C}$ be the ring of holomorphic 
functions on $\C$. Define a group homomorphism 
$\phi\colon \G \to \mathcal{O}_{\C}^{\times}$, with 
values in the units of $\mathcal{O}_{\C}$, 
by $\phi(\gamma)(t)=\exp(t \iota(\gamma))$, 
and extend $\phi$ linearly to a ring morphism 
$\bar\phi\colon \C\G \to \mathcal{O}_{\C}$.
The functions $\exp(t\iota(\gamma_1)), \dots, 
\exp(t\iota(\gamma_m))$ are linearly independent 
over $\C$, for distinct $\gamma_1, \dots, \gamma_m$; 
hence, the map $\bar\phi$ is injective. For a fixed $t\in \C$,  
let $\ev_t\colon \mathcal{O}_{\C} \to \C$ be the evaluation map. 

Since $R=\rat$ is a PID, we may decompose the 
finitely generated $R$-module $H_{i} (X, R)$ 
as $R^{b_i}\oplus \Tors_i$, where $b_i=b_i(X,\chi)$. 
Write $\bigoplus_{i=0}^{q} \Tors_i  = 
\bigoplus_{\alpha} R/f_{\alpha}R$, for some finite collection 
of non-zero Laurent polynomials $f_{\alpha}\in \Z\G$. 

($\Leftarrow$) 
Suppose $b_i=0$, for all $i\le q$.  Define
\begin{equation}
\label{eq:disc1}
D_0:=\bigcup_{\alpha} V(\bar\phi (f_{\alpha})) \cup 
\bigcup_{\text{$p: p$ is $\iota$-monic}}  V(\bar\phi (p)). 
\end{equation}

Since $\bar\phi$ is injective, $D_0$ is a countable subset of $\C$.  
Pick $t\in \C\setminus D_0$.  Since 
$t\not\in \bigcup_{p} V(\bar\phi (p))$, 
there is a ring map $\beta\colon R\to \C$ 
making the following diagram commute.  
\begin{equation}
\label{eq:holo}
\entrymodifiers={=<1.75pc>}
\xymatrix{
\Z{\Gamma}\ar^{\bar\phi}[r] \ar@{^{(}->}[d] & 
\mathcal{O}_{\C}\ar^{\ev_t}[d] \\
\rat \ar@{-->}^{\beta}[r] & \C
 }
\end{equation}

Consider the local coefficient system $\C_{\rho}$ 
on $X$, given by the character $\rho\colon G\to \C^{\times}$, 
$\rho(g)=\exp(t\chi(g))$.  Fix an integer $i\le q$.  
By the Universal Coefficient Theorem and diagram \eqref{eq:holo}, 
\begin{equation}
\label{eq:utc}
H_i(X,\C_{\rho}) = \Tor^R_0(H_i(X,R),\C_{\beta}) \oplus 
\Tor^R_1(H_{i-1}(X,R),\C_{\beta}), 
\end{equation}
where $\C_{\beta}$ denotes $\C$ viewed as an 
$\rat$-module, via $\beta$.

Since $t\not\in \bigcup_{\alpha} V(\bar\phi (f_{\alpha}))$,  
both summands in \eqref{eq:utc} vanish. 
Hence, $\rho\not\in \VV^i_1(X,\C)$, and so 
$\chi\notin \tau^{\R}_1(\VV^i_1(X,\C))$.  

($\Rightarrow$)
For the converse, define
\begin{equation}
\label{eq:disc2}
D_1:=D_0 \cup 
\psi^{-1} \Big( \bigcup_{i\le q} \VV^i_1(X,\C)\Big),
\end{equation}
where $\psi\colon \C\to \T_G$ is given by $\psi(t)= \rho$, 
with $\rho(g)=\exp(t\chi(g))$ as before.  Notice that
\[
\begin{aligned}
\chi\notin \tau^{\R}_1\Big(\bigcup_{i\le q} \VV_1^i(X, \C)\Big)  
& \same 
\{ \exp(t\chi) \mid t\in \C\} \not\subseteq \bigcup_{i\le q} \VV_1^i(X, \C)
\\
& \same   
\psi^{-1} \Big( \bigcup_{i\le q} \VV^i_1(X,\C)\Big) \ne \C.
\end{aligned}
\]

The set $\psi^{-1} \big( \bigcup_{i\le q} \VV^i_1(X,\C)\big)$ is the 
zero locus of a finite set of global holomorphic functions on 
$\C$; thus, it must be countable.  Hence, $D_1$ is countable. 

Now pick $t\in \C\setminus D_1$, and  fix an integer $i\le q$. 
Since $t\notin \bigcup_{p} V(\bar\phi (p))$, we may decompose 
$H_i(X,\C_{\rho})$ as in \eqref{eq:utc}.  Since 
$\rho=\psi(t)\notin \VV^i_1(X,\C)$, the homology group 
$H_i(X,\C_{\rho})$ must vanish.  Since $H_i(X,R)=R^{b_i}\oplus \Tors_i$, 
the right side of \eqref{eq:utc} contains $\C^{b_i}$ as a direct summand.
Hence, $b_i=0$, and we are done. 
\end{proof}

\subsection{Upper bounds for the $\Sigma$-invariants}
\label{subsec:upper}
We are now ready to state and prove our main results, partly 
summarized in the Introduction as Theorem \ref{intro:thm1}.  
Each of these results gives an upper bound on the BNSR invariants 
of a space, or of a group, starting in all generality, and then 
under certain hypotheses.  As these hypotheses get more 
specific, the upper bound gets smaller. 

\begin{theorem}
\label{thm:main}
Fix an integer $k\ge 1$. 
\begin{enumerate}
\item \label{main1}
Let $X$ be a connected CW-complex with finite $k$-skeleton.
Then, for every $q\le k$, 
\begin{equation}
\label{eq=25s}
\Sigma^q(X, \Z)\subseteq \Big(\tau_1^{\R}\big(\bigcup_{i\le q} 
\VV_1^i(X, \C)\big)\Big)^{\compl}.
\end{equation}

\item  \label{main2}
Let $G$ be a group of type $\FF_k$.  Then, for every $q\le k$, 
\begin{equation}
\label{eq:taubd}
\Sigma^q(G, \Z)\subseteq \Big(\tau_1^{\R}\big(\bigcup_{i\le q} 
\VV_1^i(G, \C)\big)\Big)^{\compl}.
\end{equation}
\end{enumerate}
\end{theorem}

\begin{proof}
Follows from Proposition \ref{thm:nov betti}, Theorem \ref{thm:nov-tau}, 
Lemma \ref{lem:exp tc}\eqref{t1}, and the remarks made after 
Definition \ref{def=bnsrx}.
\end{proof}

\begin{corollary}
\label{cor:bns tc}
With notation as in Theorem \ref{thm:main},
suppose that,  for each $i\le q$, all the components of 
$\VV^i_1(X,\C)$ (respectively, $\VV^i_1(G,\C)$), passing
through $1$ are subtori. Then, 
\begin{equation}
\label{eq:tcbd}
\Sigma^q(X, \Z)\subseteq \Big( \bigcup_{i\le q} 
TC^{\R}_1\big(\VV^i_1(X, \C)\big)\Big)^{\compl}, \  \text{resp.,} \quad
\Sigma^q(G, \Z)\subseteq \Big( \bigcup_{i\le q} 
TC^{\R}_1\big(\VV^i_1(G, \C)\big)\Big)^{\compl}.
\end{equation}
\end{corollary}

\begin{proof}
Follows from Theorem \ref{thm:main}, Lemma \ref{lem:exp tc}\eqref{t4}, and 
Proposition \ref{prop:tctc}\eqref{tc2}.
\end{proof}

\begin{corollary}
\label{cor:bns res}
With notation as in Theorem \ref{thm:main}, 
suppose the exponential formula \eqref{eq:exp} 
holds for $X$ (respectively, $G$), up to degree $q$, 
and for depth $d=1$. Then, 
\begin{equation}
\label{eq:resbd}
\Sigma^q(X, \Z)\subseteq \Big( \bigcup_{i\le q} 
\RR_1^i(X, \R)\Big)^{\compl}, \ \text{resp.,} \quad
\Sigma^q(G, \Z)\subseteq \Big( \bigcup_{i\le q} 
\RR_1^i(G, \R)\Big)^{\compl}.
\end{equation}
\end{corollary}

\begin{proof}
Follows from Theorem \ref{thm:main}, Proposition \ref{prop:tc formula},
and formula \eqref{eq:res ext}.
\end{proof}

Note that all upper bounds \eqref{eq=25s}--\eqref{eq:resbd} are 
symmetric.  In particular, they cannot be sharp when $\Sigma^q$ 
is not symmetric, as happens for the Baumslag--Solitar group 
from Example \ref{ex:1rel}.

\section{$\Sigma$-invariants and jumping loci in arbitrary characteristic}
\label{sect:arbch}

In this section, we relate the BNSR invariants of a space to 
its jump loci for homology with rank $1$ coefficients in an 
arbitrary field $\k$. Throughout, $X$ will be a connected 
CW-complex with finite $k$-skeleton $(k\ge 1)$, and 
$G=\pi_1(X)$ will be its fundamental group. 

\subsection{Valuations}
\label{ss=val}

We begin by constructing elements in the complement of 
$\Sigma^q(X,\Z)$, starting from multiplicative, $\k$-valued 
characters belonging to the characteristic varieties of $X$.  
To achieve this, we use $\R$-valuations on $\k$, 
following an idea of Bieri and Groves \cite{BGr}. 

The next result proves Part \eqref{d1}
of Theorem \ref{intro:thm2} from the Introduction. 

\begin{theorem}
\label{thm=delq}
Let $\rho\colon G\to \k^{\times}$ be a multiplicative character 
such that $\rho \in \bigcup_{i\le q} \VV^i_1(X, \k)$, for some 
$q\le k$. Let $v\colon \k^{\times}\to \R$ be the homomorphism 
defined by a valuation $v$ on $\k$, and set $\chi=v\circ \rho$.  
If the additive character $\chi\colon G\to \R$ 
is non-zero, then $\chi \not\in \Sigma^q(X, \Z)$.
\end{theorem}

\begin{proof}
Endow $\Z{G}$ with the topological ring structure defined 
by the character $\chi$, as explained in \S\ref{subsec:nov}.  
Pick a real number $a\in (0, 1)$, and put on $\k$ the 
topology associated to the absolute value $a^v$; 
in this way, $\k$ becomes a topological ring.  
Let $\widehat{\k}$ be the (topological) completion of $\k$. 
Then $\widehat{\k}$ is a field, and the structural morphism, 
$\kappa_{\k}\colon \k \inj \widehat{\k}$, is a field extension.

Let $\bar\rho\colon \Z{G}\to \k$ be the ring morphism 
defined by $\rho$. It is easy to check that $\bar\rho$ 
is continuous with respect to the above topologies.  
Hence, $\bar\rho$ induces a morphism of topological rings, 
$\hat{\rho}$, between the respective completions,  
making the following diagram commute:
\begin{equation}
\label{eq:hat rho}
\entrymodifiers={=<2pc>}
\xymatrixcolsep{28pt}
\xymatrix{
\Z{G}\ar^{\bar\rho}[r] \ar^{\kappa_{\chi}}@{^{(}->}[d] 
& \k \ar^{\kappa_{\k}}@{^{(}->}[d]\\
 \widehat{\Z{G}}_{-\chi} \ar^{\hat\rho}[r] & \widehat{\k}
 }
\end{equation}

Now suppose $\chi \in \Sigma^q(X, \Z)$. Changing coefficients 
from $ \widehat{\Z{G}}_{-\chi}$ to $\widehat{\k}$ via $\hat{\rho}$, 
we find that $H_i(X, \widehat{\k}_{\rho})=0$, for all $i\le q$. 
By Lemma \ref{lem:char}, $\rho \not\in \bigcup_{i\le q} \VV^i_1(X, \k)$,
a contradiction.
\end{proof}

\begin{remark}
\label{rem=delext}
For a finitely generated group $G$, Theorem \ref{thm=delq}---%
when applied to the classifying space $X=K(G, 1)$, with $q=k=1$, 
and $v$ a discrete valuation---recovers Proposition 1 from \cite{Dz}, 
proved by Delzant using a different method, based on 
$G$-actions on trees.
\end{remark}

\subsection{Algebraic integers}
\label{ss=algint}

Next, we examine the converse question: how to produce 
multiplicative $\k$-characters with trivial (twisted) homology, 
starting from elements in $\Sigma^q$. We will need the 
following lemma.

\begin{lemma}
\label{lem=torann}
Let $\Lambda$ be a commutative ring, $\k$ a field, 
and $\rho\colon \Lambda \to \k$ a ring morphism. If 
$\Tor^{\Lambda}_s (\k, N)\ne 0$, for some $\Lambda$-module 
$N$ and some $s\ge 0$, then $\rho(\Delta)=0$, for  every 
$\Delta \in \ann_{\Lambda}(N)$.
\end{lemma}

\begin{proof}
Consider the homothety $\mu_{\Delta}$, acting $\Lambda$-linearly on 
$\Lambda$-modules. Since $\Delta$ annihilates $N$, this homothety 
acts as the zero map on $\Tor^{\Lambda}_s (\k, N)$. 

Now pick a free $\Lambda$-resolution, $P_{\bullet}\surj N$, 
where $P_{\bullet}=\Lambda \otimes_{\Z} C_{\bullet}$. Clearly, 
$\mu_{\Delta}$ acts on $\Tor^{\Lambda}_s (\k, N)= 
H_s(\k \otimes_{\Z} C_{\bullet})$ as the $\k$-homothety 
$\mu_{\rho(\Delta)}$. Since $\Tor^{\Lambda}_s (\k, N)\ne 0$
and $\k$ is a field, we must have $\rho(\Delta)=0$.
\end{proof}

As before, let $X$ be a connected CW-complex with finite 
$k$-skeleton $(k\ge 1)$, and $G=\pi_1(X)$.
Given a non-trivial additive character $\chi\in \Hom(G, \R)$, 
factor it as $\chi= \iota \circ \xi$, where $\xi\colon G \surj \Gamma$ 
is the co-restriction of $\chi$ onto $\Gamma:= \im (\chi)$, and 
$\iota\colon \Gamma \hookrightarrow \R$ is the inclusion. 

A multiplicative character $\rho\colon  \Gamma \to \k^{\times}$ 
is called an {\em algebraic integer}\/ if $\Delta (\rho)=0$, for some 
Laurent polynomial $\Delta = \sum_{\gamma} n_{\gamma} \gamma 
\in \Z\Gamma$ with the property that $n_{\gamma_0}=1$, where 
$\gamma_0$ is the greatest element of $\iota(\supp (\Delta))$. 
Compare with \S\ref{subsec:rat nov}, and \cite[Definition 1.53]{Fa}. 

The next result---inspired in part by Theorem 1.54 from \cite{Fa}---% 
completes the proof of Theorem \ref{intro:thm2} from the Introduction. 

\begin{theorem}
\label{thm=delq converse}
Let $\chi=\iota\circ \xi\colon G\to \R$ be an additive character 
such that $-\chi \in \Sigma^q(X, \Z)$, for some $q\le k$.  
Let $\rho\colon  \Gamma \to \k^{\times}$ be a multiplicative 
character which is not an algebraic integer.
Then $\xi^* \rho \not\in \bigcup_{i\le q} \VV^i_1(X, \k)$.
\end{theorem}

\begin{proof}
By Proposition \ref{thm:nov betti}, $H_i(X, \rat)=0$, for all $i\le q$.
In view of Lemma \ref{lem:skeleton}, we may assume, without 
loss of generality, that $X$ is a finite CW-complex, since the 
ring $\rat$ is commutative.
By \cite[Proposition 1.34(2)]{Fa}, there is an element 
$\Delta = \sum_{\gamma} n_{\gamma} \gamma$  in 
the ring $\Lambda:= \Z \Gamma$, so that $\Delta$ has 
top coefficient $n_{\gamma_0}=1$, and $\Delta$ annihilates 
$H_i(X, \Lambda)$, for all $i\le q$.

Consider now the spectral sequence associated to the change 
of rings map $\bar\rho\colon \Lambda \to \k$,
\begin{equation}
\label{eq:ss rho}
E^2_{st}= \Tor^{\Lambda}_s (\k_{\rho}, H_t(X, \Lambda)) \Rightarrow 
H_{s+t}(X, \k_{\xi^* \rho}).
\end{equation}
Since, by assumption, the character $\rho$ is not an algebraic 
integer, Lemma \ref{lem=torann} implies that $E^2_{st}=0$, for 
all $t\le q$ and $s\ge 0$.  Consequently, 
$H_{i}(X, \k_{\xi^* \rho})=0$, for all $i\le q$, showing that 
$\xi^* \rho \not\in \bigcup_{i\le q} \VV^i_1(X, \k)$.
\end{proof}

\section{$\Sigma$-invariants, $1$-formality, and deficiency}
\label{sec:apps}

For the remainder of this paper, we will give various examples 
for which the bounds derived in  Section \ref{sec:main} hold, some 
with equality.  We start in this section with a particularly manageable 
case---that of $1$-formal groups---and discuss the relationship between 
$\Sigma$-invariants and deficiency.  

\subsection{$\Sigma$-invariants of $1$-formal groups}
\label{subsec:formal bound}
Recall from Theorem \ref{thm:dpsA} that the exponential formula 
holds in degree $1$, whenever $G$ is a $1$-formal group.  It follows 
from Proposition \ref{prop:tc formula} that $\RR^1_d(G, \C)$ 
is the union of a rationally defined subspace arrangement. 
Corollary \ref{cor:bns res} then yields the following.

\begin{prop}
\label{prop:1formal}
Suppose $G$ is a $1$-formal group. Then the BNS invariant 
of $G$ is contained in the complement of a finite union of 
linear subspaces defined over $\Q$:
\begin{equation}
\label{eq:formal bd}
\Sigma^1(G)\subseteq \RR_1^1(G, \R)^{\compl}.
\end{equation}
\end{prop}

A nice example where the above inclusion holds as an equality is 
provided by the McCool groups of basis-conjugating automorphisms, 
see D.~Cohen \cite{Co}. 

In general, though, inclusion \eqref{eq:formal bd} is strict.  For instance, 
take $G$ to be the Baumslag-Solitar group from Example \ref{ex:1rel}. 
Then $G$ is $1$-formal, since $b_1(G)=1$; yet $\Sigma^1(G)
\ne \RR_1^1(G, \R)^{\compl}$, since $\Sigma^1(G)$ is not 
symmetric about the origin. 

In favorable situations, Proposition \ref{prop:1formal} 
allows us to determine exactly the BNS invariant of a group. 

\begin{corollary}
\label{cor:empty}
Suppose $G$ is a $1$-formal group, and $\RR_1^1(G, \R)=H^1(G,\R)$. 
Then $\Sigma^1(G)=\emptyset$.   
\end{corollary}

This approach yields a quick application to $3$-manifolds. 

\begin{corollary}
\label{cor:3mfd}
Let $M$ be a closed, orientable $3$-manifold, and let $G=\pi_1(M)$.   
Suppose $b_1(M)$ is even, and $G$ is $1$-formal. 
Then $\Sigma^1(G)=\emptyset$ (and thus, $M$ does 
not fiber over $S^1$).  
\end{corollary}
 
\begin{proof}
Since $b_1(M)$ is even, we have $\RR^1_1(G,\C)=H^1(G,\C)$,  
by \cite[Proposition 4.2]{DS}. The conclusion follows 
from Corollary \ref{cor:empty}.
\end{proof}

If $b_1(M)$ is odd, the $3$-manifold $M$ may of course 
fiber over the circle, even if $G$ is $1$-formal, but there 
are restrictions on what the algebraic monodromy can be. 
These aspects are pursued in \cite{PS-mono}, in a more 
general setting. 

\subsection{$\Sigma$-invariants and deficiency}
\label{subsec:def}
Suppose $G$ is a finitely presented group.  Its {\em deficiency}, 
denoted $\df(G)$, is the maximum value of the difference between 
the number of generators and the number of relations, taken over 
all finite presentations of $G$.  

It turns out that there is a 
strong connection between the $\Sigma$-invariants of a 
group and its deficiency.  For instance, Bieri, Neumann, and 
Strebel proved in \cite[Theorem 7.2]{BNS} the following: if 
$\df(G)\ge 2$, then $\Sigma^1(G)= \emptyset$.
In \cite{Br}, K.~Brown gave an explicit algorithm for 
computing the BNS invariant of a $1$-relator group $G$. 
In \cite{BR}, Bieri and Renz reinterpreted this algorithm 
in terms of Fox calculus, and showed that, for $1$-relator 
groups, $\Sigma^{r}(G,\Z)=\Sigma^1(G)$, for all $r\ge 2$. 
For more on this subject, see Bieri \cite{Bi07}. 

Now let $X$ be a connected CW-complex with finite $2$-skeleton, 
and denote by $c_i$ the number of $i$-cells of $X$, for $i=1,2$. 
The proposition below provides a simple class of examples where 
our general upper bound is optimal. For groups with deficiency 
at least $2$, this proposition recovers the result of 
Bieri--Neumann--Strebel mentioned above. 

\begin{prop}
\label{prop:sdef}
If $c_1-c_2>1$, then 
$\Sigma^1(X, \Z)= \big( \tau_1^{\R} \big( 
\VV_1^1(X, \C)\big) \big)^{\compl}= \emptyset$. 
\end{prop}

\begin{proof}
Let $\rho\in \Hom (\pi_1(X), \C^{\times})$ be a character. 
The chain complex computing the homology of 
$X$ with coefficients in $\C_{\rho}$ ends in 
$\C^{c_2} \to \C^{c_1} \to \C$. Therefore, 
$\dim_{\C} H_1(X, \C_{\rho})$ is bounded below by 
$c_1-c_2-1>0$, and so  $\rho\in \VV^1_1(X, \C)$. 
The conclusion follows from 
Theorem \ref{thm:main}. 
\end{proof}

\section{Finitely generated nilpotent groups}
\label{sec:nilp}

In this section, we delineate another large class of groups
for which the ``exponential tangent cone upper bound" 
for the $\Sigma$-invariants is, in fact, an equality. 

\subsection{Finiteness properties}
\label{subsec:fin nilp}

A group $G$ is said to be {\em nilpotent}\/ if its 
lower central series, defined inductively by $\gamma_1G= G$ 
and $\gamma_{i+1}G=[\gamma_iG,G]$ for $i\ge 1$, where 
$[\:,\: ]$ stands for the group commutator, terminates in 
finitely many steps. Clearly, subgroups and factor groups 
of nilpotent groups are again nilpotent. 

\begin{lemma}
\label{lem:nilp fp}
Let $G$ be a finitely generated nilpotent group.  
Then $G$ is of type $\FP_{\infty}$.
\end{lemma}

\begin{proof}
Let $\Tors(G)$ be the  torsion subgroup of $G$.  
From the assumptions, it follows that $\Tors(G)$ is finite, 
hence of type $\FF_{\infty}$.  The quotient group, 
$Q=G/\Tors(G)$, is a torsion-free, finitely 
generated nilpotent group.  As such, $Q$ admits 
as classifying space a compact nilmanifold of the form 
$M=\R^n/Q$; in particular, $Q$ is of type $\FF$. 
By Lemma \ref{lem:bieri}, $G$ is of type $\FP_{\infty}$. 
\end{proof}

\subsection{Jumping loci}
\label{subsec:jump nilp}

The characteristic and resonance varieties of finitely 
generated nilpotent groups were studied in \cite{MP}.  
The former admit a very simple description. 

\begin{theorem}[\cite{MP}]
\label{thm:cv nilp}
Let $G$ be a finitely generated nilpotent group.  Then 
\begin{equation*}
\VV^i_d(G,\C)=
\begin{cases} 
\{1\} & \text{if $b_i(G)\ge d$},\\[2pt]
\emptyset & \text{otherwise}.
\end{cases}
\end{equation*}
\end{theorem}

On the other hand, the resonance varieties of such groups  
can be just about as complicated as possible. For example, 
as noted in \cite[Remark 2.4]{MP}, given {\em any}\/ 
finitely generated group $H$, there is a $2$-step 
nilpotent group $G$ such that $\RR^1_d(G,\C)=\RR^1_d(H,\C)$, 
for all $d\ge 1$: simply take $G=H/\gamma_3H$.

\subsection{$\Sigma$-invariants}
\label{subsec:bnsnilp}
For completeness, we include a proof for the following result, 
which is presumably folklore. 

\begin{prop}
\label{prop:bns nilp}
Let $G$ be a finitely generated nilpotent group.  Then 
\begin{equation*}
\Sigma^q(G,\Z)=H^1(G,\R)\setminus \{0\}, 
\quad\text{for all $q\ge 0$}. 
\end{equation*}
\end{prop}

\begin{proof}
Any subgroup $H<G$ is again a finitely generated 
nilpotent group.  By Lemma \ref{lem:nilp fp}, $H$ is 
of type $\FP_{\infty}$.  The claim follows at once 
from Theorem \ref{thm:bnsr finp}\eqref{bf2}. 
\end{proof}

Comparing the two answers from Theorem \ref{thm:cv nilp} 
and Proposition \ref{prop:bns nilp}, we see that, in this case, 
the inclusion guaranteed by Theorem \ref{thm:main} is actually an 
equality. 

\begin{corollary}
\label{cor:match nilp}
Let $G$ be a finitely generated nilpotent group.  
Then 
\begin{equation*}
\Sigma^q(G,\Z)=\Big(\tau^{\R}_1 \big( \bigcup_{i\le q} 
\VV^i_1(G,\C)\big)\Big)^{\compl}, 
\quad\text{for all $q\ge 0$}. 
\end{equation*}
\end{corollary}

On the other hand, the ``resonance upper bound" from 
Proposition \ref{prop:1formal} does not hold for all 
groups in this class, as the next example shows.

\begin{example}
\label{ex:heisenberg}
Let $G=F_2/\gamma_3 F_2$ be the Heisenberg group. 
It is readily seen that the cup-product map on $H^1(G,\Z)=\Z^2$ 
is zero, and so $\RR^1_1(G,\R)=\R^2$.  In view of Proposition 
\ref{prop:bns nilp}, the inclusion 
$\Sigma^1(G)\subseteq \RR^1_1(G, \R)^{\compl}$ 
completely fails in this case.  This observation recovers the 
well-known fact that $G$ is {\em not}\/ a $1$-formal group. 
\end{example}

\section{Products and coproducts} 
\label{sec:products}

In this section, we analyze how the various invariants 
we studied so far behave under the product and coproduct 
operations for spaces and groups, and the extent to which 
our bounds for the $\Sigma$-invariants carry over under 
those operations. 

\subsection{Jumping loci of products}
\label{subsec:cjl prod}
Let $X_1$ and $X_2$ be two CW-complexes, each one 
with a single $0$-cell (which we will take as the basepoint), 
and finitely many $1$-cells. Denote by $G_1$ and $G_2$ 
the corresponding fundamental groups.  Fix a field $\k$;   
if $\ch \k=2$, assume $H_1(X_j, \Z)$ is torsion free, 
to have the resonance varieties of $X_1$ and $X_2$ defined.  
The jumping loci of the product CW-complex $X_1\times X_2$ 
can be computed in terms of the jumping loci of the factors, 
as follows. 

\begin{prop}
\label{prop:cjl prod}
For $i\le \min (\varkappa(X_1), \varkappa(X_2))$, 
\begin{romenum}
\item \label{prodv}
$\VV^i_1(X_1\times X_2,\k)= 
\bigcup_{s+t=i} \VV^{s}_1(X_1,\k) \times \VV^{t}_1(X_2,\k)$.
\item \label{prodr}
$\RR^i_1(X_1\times X_2,\k)=
\bigcup_{s+t=i} \RR^{s}_1(X_1,\k) \times \RR^{t}_1(X_2,\k)$.
\end{romenum}
\end{prop}

\begin{proof}
\eqref{prodv}   
Write $X=X_1\times X_2$, and identify $G:=\pi_1(X)$ with 
$G_1\times G_2$.  
Let $\wX=\wX_1\times \wX_2$ be the universal cover 
of $X$. We then have a $G$-equivariant isomorphism 
of chain complexes, $C_{\bullet}(\wX)\cong  
C_{\bullet}(\wX_1)\otimes C_{\bullet}(\wX_2)$. 
Given a character $\rho=(\rho_1,\rho_2)$ in 
$\Hom(G,\k^{\times})=\Hom(G_1,\k^{\times})
\times \Hom(G_2,\k^{\times})$, we obtain an isomorphism  
$C_{\bullet}(X,\k_{\rho}) \cong C_{\bullet}(X_1,\k_{\rho_1})
\otimes_{\k} C_{\bullet}(X_2,\k_{\rho_2})$. 
Taking homology, we find 
$H_i(X,\k_{\rho})=\bigoplus_{s+t=i} 
H_{s}(X_1,\k_{\rho_1}) \otimes_{\k} H_{t}(X_2,\k_{\rho_2})$, 
and the claim follows. 

\eqref{prodr} 
Write $A=H^*(X,\k)$ and $A_j=H^*(X_j,\k)$, for $j=1,2$. 
By the K\"{u}nneth formula, 
$A=A_1\otimes_{\k} A_2$. Pick an element $a=(a_1, a_2)$ 
in $A^1=A^1_1\oplus A^1_2$.  The Aomoto 
complex of $X$ splits as a tensor product of chain complexes, 
$(A,a) \cong (A_1,a_1) \otimes_{\k} 
(A_2,a_2)$. The claim follows. 
\end{proof}

Various particular cases of this Proposition are known 
from the literature, see e.g.~ Corollary 3.3.2 in \cite{Hi}, 
Theorem 3.2 in \cite{CS}, Lemma 7.5 in \cite{DPS-serre},
and Proposition 5.7 in \cite{Mac}.

The next lemma shows that exponential tangent cones also 
behave well with respect to direct products; the proof is left 
as an exercise in the definitions. 
 
\begin{lemma}
\label{lem:tauprod}
Let $G=G_1\times G_2$ be a product of finitely generated groups, 
and let $W_1\subseteq \T_{G_1}$ and $W_2\subseteq \T_{G_2}$ 
be Zariski closed subsets.  Then,
$\tau_1(W_1\times W_2)= \tau_1(W_1)\times \tau_1(W_2)$. 
\end{lemma}

\subsection{Jumping loci of wedges}
\label{subsec:cjl wedges}

We now express the jumping loci of the wedge 
$X_1\vee X_2$ in terms of the jumping loci of the factors. 

\begin{prop}
\label{prop:cjl vee}
Suppose $b_1(X_j, \k)\ne 0$, for $j=1,2$. Then, 
for $ i\le \min (\varkappa(X_1), \varkappa(X_2))$,  
\begin{romenum}
\item \label{veev}
$\VV^i_1(X_1\vee X_2,\k) = \begin{cases}
\Hom(G_1,\k^{\times})\times \Hom(G_2,\k^{\times}) 
& \text{if $i=1$,}\\[2pt]
\VV^i_1(X_1,\k)\times \Hom(G_2,\k^{\times})  
\cup 
\Hom(G_1,\k^{\times}) \times \VV^i_1(X_2,\k)
&\text{if $i>1$.}
\end{cases}$
\item \label{veer}
$\RR^i_1(X_1\vee X_2,\k)= \begin{cases}
H^1(X_1,\k) \times H^1(X_2,\k) 
&\text{if $i=1$,}\\[2pt]
\RR^i_1(X_1,\k)\times H^1(X_2,\k)   
\cup 
H^1(X_1,\k)  \times \RR^i_1(X_2,\k)
 &\text{if $i>1$.}
\end{cases}$
\end{romenum}
\end{prop}

\begin{proof}
\eqref{veev} 
Write $X=X_1\vee X_2$, and identify $G:=\pi_1(X)$ 
with $G_1* G_2$.  Given a character $\rho=(\rho_1,\rho_2)$ in 
$\Hom(G,\k^{\times})=\Hom(G_1,\k^{\times})
\times \Hom(G_2,\k^{\times})$, we have an isomorphism 
of chain complexes, 
$C_{+}(X,\k_{\rho}) \cong C_{+}(X_1,\k_{\rho_1})
\oplus C_{+}(X_2,\k_{\rho_2})$. Write 
$b_i(X,\rho):=\dim_{\k} H_i(X,\k_{\rho})$.  We then have:
\[
b_i(X,\rho)=\begin{cases}
b_i(X_1,\rho_1)+b_i(X_2,\rho_2)+1&\text{if $i=1$, 
and $\rho_1\ne 1$, $\rho_2\ne 1$,}\\
b_i(X_1,\rho_1)+b_i(X_2,\rho_2)&\text{otherwise.}
\end{cases}
\]
Since $b_1(X_j,1)=b_1(X_j)>0$, the claim follows. 

\eqref{veer} 
With notation as before, pick an element $a=(a_1, a_2)$ in 
$A^1=A^1_1\oplus A^1_2$.  The Aomoto complex of $X$ 
splits (in positive degrees) as a direct sum of chain complexes, 
$(A^{+},a) \cong (A^{+}_1,a_1) \oplus (A^{+}_2,a_2)$. 
We then have:
\[
\beta_i(A,a)=\begin{cases}
\beta_i(A_1, a_1)+\beta_i(A_2, a_2)+1&\text{if $i=1$, 
and $a_1\ne 0$, $a_2\ne 0$,}\\
\beta_i(A_1, a_1)+\beta_i(A_2, a_2)&\text{otherwise.}
\end{cases}
\]
The claim follows. 
\end{proof}

Various particular cases of this Proposition are known 
from the literature, see e.g.~ Proposition 3.2.1 in \cite{Hi}, 
Corollary 5.4 in \cite{PS-artin}, and Lemma 7.8 in \cite{DPS-serre}.  

\subsection{$\Sigma$-invariants of products and coproducts}
\label{subsec:bnsprod}

The following product formula for the BNSR invariants was 
recently established by Bieri and Geoghegan in \cite{BiG}, 
building on work on several others.   Their result, slightly 
modified to fit our setup, reads as follows.  

\begin{theorem}[\cite{BiG}]
\label{thm:bg}
Let $G_1$ and $G_2$ be two groups of type $\FF_k$ ($k\ge 1$). 
For all $q\le k$, the following hold:
\begin{romenum}
\item \label{bnsp1}
$\Sigma^q(G_1\times G_2,\Z)^{\compl} \subseteq
\bigcup_{i+j=q} \Sigma^{i}(G_1,\Z)^{\compl} \times 
\Sigma^{j}(G_2,\Z)^{\compl}$, with equality for $q\le 3$. 
\item \label{bnsp2}
$\Sigma^q(G_1\times G_2,\k)^{\compl}=
\bigcup_{i+j=q} \Sigma^{i}(G_1,\k)^{\compl} \times \Sigma^{j}(G_2,\k)^{\compl}$, 
for all fields $\k$.
\end{romenum}
\end{theorem}

The above product formula was established for $q=1$ 
by Bieri, Neumann, and Strebel in \cite[Theorem 7.4]{BNS}, 
and was conjectured to hold in general (over $\Z$) by 
Bieri \cite{Bi99}.  The inclusion from \eqref{bnsp1} was 
established by Meinert (unpublished), with a proof given 
by Gehrke \cite{Ge} (see also Bieri \cite{Bi99}).  The fact 
that equality holds in \eqref{bnsp1} for $q\le 3$ was proved 
by Sch\"{u}tz \cite{Sc} and Bieri--Geoghegan \cite{BiG}.  
For $q\ge 4$, equality may fail in \eqref{bnsp1}, 
as examples from \cite{Sc} show. 

As is well-known, the $\Sigma$-invariants of a non-trivial 
free product are all empty.  For completeness, we include 
a proof of this fact. 

\begin{prop}
\label{prop:bns coprod}
Let $G_1$ and $G_2$ be two non-trivial, finitely generated 
groups. Then $\Sigma^q(G_1*G_2,\Z)=\emptyset$, for all $q\ge 1$. 
\end{prop}

\begin{proof}
This follows at once from a result of Bieri and Strebel, 
as recounted in \cite[Lemma 2.4]{Me}: If $G=G_1 *_{K} G_2$ 
is a non-trivial free product with amalgamation, and 
$\chi\colon G\to \R$ is a homomorphism such that 
$\left.\chi\right|_K=0$,  then $\chi \notin \Sigma^1(G)$.  
In our situation, $K=\{1\}$; thus, $\Sigma^1(G)=\emptyset$, 
and so $\Sigma^q(G,\Z)=\emptyset$, for all $q\ge 1$.
\end{proof}

\subsection{Upper bounds for products and coproducts}
\label{subsec:up prod}
We now analyze the extent to which our bounds for 
the $\Sigma$-invariants carry over under direct products. 
We start by showing that, if the exponential tangent cone 
upper bound \eqref{eq:taubd} holds as an equality for 
two groups, up to some degree $q$, then it also holds 
as an equality for their product, in degree $q$. 

\begin{theorem}
\label{thm:bound prod}
Let $G_1$ and $G_2$ be two groups of type $\FF_k$ ($k\ge 1$). 
For each $q\le k$, the following holds:
\begin{align*}
\Sigma^r(G_j, \Z)&= \Big(\tau_1^{\R}\big(\bigcup_{i\le r} 
\VV_1^i(G_j, \C)\big)\Big)^{\compl},\quad  \text{for $j=1,2$ and $r\le q$} \\
&\qquad \implies 
\Sigma^q(G_1\times G_2, \Z)= \Big(\tau_1^{\R}\big(\bigcup_{i\le q} 
\VV_1^i(G_1\times G_2, \C)\big)\Big)^{\compl}.
\end{align*}
\end{theorem}

\begin{proof} 
Follows from Theorem \ref{thm:main}, 
Proposition \ref{prop:cjl prod}\eqref{prodv}, 
Lemma \ref{lem:tauprod}, 
and Theorem \ref{thm:bg}\eqref{bnsp1}.
\end{proof}

\begin{remark}
\label{rem:zbieri}
As a byproduct of the proof of Theorem \ref{thm:bound prod}, 
we have the following:  If the exponential tangent cone 
upper bound \eqref{eq:taubd} 
holds as an equality for $G_1$ and $G_2$, up to degree $q$, then 
the Bieri Conjecture is true integrally for $G_1\times G_2$ in degree $q$, 
i.e., equality holds in Theorem \ref{thm:bg}\eqref{bnsp1}. 
\end{remark}

As for (non-trivial) free products, we have the following observation. 

\begin{prop}
\label{prop:brcp}  
Let $G=G_1*G_2$ be a non-trivial free product of finitely generated 
groups. Then:  
\begin{romenum}
\item \label{bdv1}
The resonance upper bound \eqref{eq:resbd} holds for $G$, 
in all degrees $q\ge 0$.  

\item \label{bdv2}
If $b_1(G_1)>0$ and $b_1(G_2)>0$, then the resonance 
upper bound \eqref{eq:resbd} holds as an equality for $G$, 
in all degrees $q\ge 0$.  
\end{romenum}
\end{prop}

\begin{proof}
Follows from Propositions \ref{prop:bns coprod} and 
\ref{prop:cjl vee}\eqref{veer}. 
\end{proof}

\section{Toric complexes and right-angled Artin groups}
\label{sec:artin}

We now consider certain CW-complexes which 
interpolate between a wedge and a product of circles, 
and find that the ``resonance upper bound" for the 
BNSR invariants holds for such spaces (and their 
fundamental groups), often with equality. 

\subsection{Toric complexes and Stanley-Reisner ring}
\label{subsec:tc sr}

Every finite simplicial complex $L$ gives rise to a finite 
CW-complex, $T_L$. If $\V$ is the vertex set of $L$, and 
$T^{n}$ is the torus of dimension $n=\abs{\V}$, with the 
standard CW-decomposition, then the {\em toric complex} 
$T_L$ is the subcomplex of $T^{n}$ obtained by deleting 
the cells corresponding to the non-faces of $L$. 

The fundamental group, $G_L=\pi_1(T_L)$, coincides with 
the {\em right-angled Artin group} $G_{\G}$ defined by the 
graph $\G=L^{(1)}$. The group $G_{\G}$  has a presentation 
consisting of a generator $v$ for each vertex $v\in \sV$, 
and a commutator relation $vw=wv$ for each edge 
$\{v,w\}$ in $\G$.  A classifying space for 
$G_{\G}$ is the toric complex $T_{\Delta}$, 
where $\Delta=\Delta_{\G}$ is the flag complex of $\G$. 
In particular, every right-angled Artin group is 
of type $\FF$. 

Since $T_L$ is a subcomplex of the $n$-torus, $H_*(T_L,\Z)$ 
is torsion-free.  The cohomology ring of $T_L$, with 
coefficients in a field $\k$, is isomorphic to the exterior 
Stanley-Reisner ring $\SR$, with generators the duals 
$v^*$, and relations the monomials corresponding 
to the missing faces of $L$. 

Given an element $z\in \SR^1$, 
write $z=\sum_{v\in \sV} z_v v^*$, with $z_v\in \k$, and define 
its support as $\supp (z)=\set{v \in \V\mid z_v \ne 0}$.  
For a subset $\sW\subseteq \sV$, write 
$z_{\sW}= \sum_{v\in \sW} v^*$.  The calculation of the 
Aomoto-Betti numbers $\beta_{i}( \SR , z)$ is 
greatly simplified by the following lemma. 

\begin{lemma}[\cite{AAH}] 
\label{lem:aah}
Let $z\in  \SR^1$, and $\sW=\supp(z)$. 
Then $\beta_{i}( \SR , z)=\beta_{i}( \SR , z_{\sW})$, 
for all $i\ge 0$.
\end{lemma}

The Aomoto-Betti numbers of the Stanley-Reisner ring 
may be computed solely in terms of the simplicial 
complex $L$, by means of the following formula 
from Aramova, Avramov, and Herzog  \cite{AAH}, 
as slightly modified in \cite{PS-toric}:
\begin{equation}
\label{eq:aah}
\beta_i(\SR,z_\sW)=\sum_{\sigma\in L_{\sV\setminus \sW}}
\dim_{\k} \widetilde{H}_{i-1-\abs{\sigma}} (\lk_{L_\sW}(\sigma),\k).
\end{equation}
Here $L_\sW=\set{\tau\in L\mid \tau \subseteq \sW}$ is 
the induced subcomplex of $L$ on vertex set $\sW$, 
and $\lk_{K}(\sigma)= 
\set{\tau \in K \mid \tau\cup \sigma \in L}$ is the 
link in a subcomplex $K$ of a simplex $\sigma$.

\subsection{Jumping loci of toric complexes}
\label{subsec:jump tc}

The resonance and characteristic varieties of toric 
complexes were computed in \cite{PS-toric}, extending 
previous results from \cite{PS-artin} and \cite{DPS-serre}.  

The cohomology group $H^1(T_L,\k)=\SR^{1}$ may be 
identified with  the $\k$-vector space with basis $\sV$, 
denoted $\k^{\sV}$.  For a subset $\sW \subseteq \sV$, let 
$\k^\sW$ be the corresponding coordinate subspace. 
The character variety of the associated right-angled Artin group, 
$\Hom(G_L,\k^{\times})$, may be identified with the algebraic torus 
$(\k^{\times})^{\sV}$.  For a subset $\sW \subseteq \sV$, let  
$(\k^{\times})^{\sW}$ be the corresponding coordinate subtorus. 

\begin{theorem}[\cite{PS-toric}]
\label{thm:cv toric}
Let $L$ be a finite simplicial complex on vertex set $\sV$, 
and let $T_L$ be the associated toric complex. Then, 
for any field $\k$, and every $i\ge 0$ and $d>0$,
\begin{romenum}
\item \label{tcv}
$\VV^i_d(T_L,\k)= \bigcup_{\sW} \, (\k^{\times})^{\sW}$,
\item \label{tres}
$\RR^i_d(T_L,\k)= \bigcup_{\sW} \, \k^{\sW}$,
\end{romenum}
where in both cases, the union is taken over all subsets 
$\sW \subseteq \sV$ for which $\beta_{i}(\SR,z_\sW)\ge d$. 
\end{theorem}

Identify the exponential map $\Hom(G_L,\C)\to 
\Hom(G_L,\C^{\times})$ with the component-wise 
exponential $\C^{\sV} \to (\C^{\times})^{\sV}$.  
Clearly, $\exp(\C^{\sW})=(\C^{\times})^{\sW}$, 
for all $\sW\subseteq \sV$. 

\begin{corollary}
\label{cor:exp raag}
For each $i\ge 0$ and $d>0$,
\begin{romenum}
\item \label{raag1}
$\exp\colon (\RR^i_d(T_L,\C),0) \isom (\VV^i_d(T_L,\C),1)$.
\item  \label{raag2}
$\tau_1 (\VV_d^i(T_L, \C))= TC_1(\VV_d^i(T_L, \C))= \RR_d^i(T_L, \C)$.
\item \label{raag3}
$\Sigma^q(T_L, \Z)\subseteq \Big( \bigcup_{i\le q} 
\RR_1^i(T_L, \R)\Big)^{\compl}, \ \text{for all $q\ge 0$}$.
\end{romenum}
\end{corollary}

\begin{proof}
Part \eqref{raag1} follows from Theorem \ref{thm:cv toric}, 
Part \eqref{raag2} is a consequence of Proposition \ref{prop:tc formula}, 
and Part \eqref{raag3} is implied by Corollary \ref{cor:bns res}.
\end{proof}

\subsection{$\Sigma$-invariants of right-angled Artin groups}
\label{subsec:raag}

The Bieri--Neumann--Strebel--Renz invariants of a right-angled 
Artin group $G_{\G}$ were computed by Meier, Meinert, and 
VanWyk in \cite{MMV}.  A more convenient---yet 
equivalent---de\-scription of these invariants was given 
by Bux and Gonzalez in \cite{BG}. 

\begin{theorem}[\cite{MMV}, \cite{BG}]
\label{thm:mmv bg}
Let $\k$ be a commutative ring.
Let  $\chi\colon G_{\G}\to \R$ be a non-zero homomorphism, 
with support $\sW=\set{v \in \sV \mid \chi(v)\ne 0}$. 
Then, $\chi \in \Sigma^{q}(G_{\G},\k)$ if and only if 
$\widetilde{H}_{j}(\lk_{L_{\sW}} (\sigma),\k)=0$, 
for every simplex $\sigma\in L_{\sV\setminus \sW}$, 
and every $-1\le j\le q-\dim(\sigma)-2$.  
\end{theorem}

Comparing the above description of $\Sigma^q(G_{\G},\k)$ to 
that of $\RR^i_1(G_{\G},\k)$ from Theorem \ref{thm:cv toric}, 
and making use of formula \eqref{eq:aah}, we obtain the 
following result. 

\begin{corollary}
\label{thm:tc bns raag}
Let $\chi\in H^1(G_{\G},\R)\setminus \set{0}$.   For any field $\k$, 
and any integer $q\ge 0$,
\begin{equation*}
\label{eq:bns res raag}
\chi\in \Sigma^q(G_{\G},\k) \same z_{\supp(\chi)} \notin 
\bigcup_{i\le q} \RR^i_1(G_{\G},\k). 
\end{equation*}
\end{corollary}

For $\k=\R$, we obtain the following corollary, showing 
that the ``resonance upper bound" predicted by 
Corollary \ref{cor:bns res} is attained for right-angled 
Artin groups, at the price of replacing $\Z$-coefficients 
by $\R$-coefficients. 

\begin{corollary}
\label{cor:res bns raag}
$\Sigma^q(G_{\G},\R) =  \big(\bigcup_{i\le q} 
\RR^i_1(G_{\G},\R)\big)^{\compl}$, for all $q\ge 0$. 
\end{corollary}

\begin{proof}
Let $\Delta=\Delta_{\G}$ be the flag complex of $\G$, 
and recall that $T_{\Delta}=K(G_{\G},1)$.   By 
Lemma \ref{lem:aah}, we have 
$\beta_i(\R\langle \Delta\rangle , \chi)=
\beta_{i}( \R\langle \Delta\rangle  , z_{\supp(\chi)})$. 
Thus, 
$\chi\in \RR^i_1(G_{\G},\R)$ if and only if 
$z_{\supp(\chi)}\in \RR^i_1(G_{\G},\R)$.  The conclusion 
now follows from Corollary \ref{thm:tc bns raag}. 
\end{proof}

When $q=1$, this Corollary recovers Proposition 5.8 from \cite{PS-artin}, 
which says that $\Sigma^1(G_{\G}) = \RR^1_1(G_{\G},\R)^{\compl}$.  
After adding appropriate conditions on the flag complex $\Delta=\Delta_\G$, 
the resonance upper bound becomes an equality 
for all $q$, even integrally.  

\begin{corollary}
\label{cor:tfree}
Suppose that, for every simplex $\sigma \in \Delta$ 
and for every $\sW\subseteq \sV$ 
such that $\sigma \cap W=\emptyset$, the subcomplex 
$\lk_{\Delta_{\sW}}(\sigma)$ has torsion-free integral 
homology (this happens, for instance, if $\G$ is a tree). 
Then:  
\[
\Sigma^q(G_{\G}, \Z)= \Sigma^q(G_{\G}, \k)= 
\Big(\bigcup_{i\le q} \RR^i_1(G_{\G}, \R)\Big)^{\compl},
\]
for all $q\ge 0$, and all fields $\k$. 
\end{corollary}

The torsion-freeness condition in the above Corollary is 
really necessary, as shown in the next example. 

\begin{example}
\label{ex:rp2}
Let $\Delta$ be a flag triangulation of the real projective plane, 
$\RP^2$, and let $\nu\colon G_{\Delta}\to \Z$ be the diagonal 
homomorphism, sending each standard generator to $1$.  
Then $\nu\not\in \Sigma^2( G_{\Delta}, \Z)$
and $\nu\notin \Sigma^2( G_{\Delta}, \k)$, if $\ch \k=2$, but
$\nu \in \Sigma^2( G_{\Delta}, \k)$, if $\ch \k\ne 2$. In particular, 
$\nu\in\Sigma^2 (G_{\Delta},\R) \setminus  
\Sigma^2( G_{\Delta}, \Z)$.
\end{example}

\section{Free abelian covers of toric complexes and Artin kernels}
\label{sec:toric cov}

We now study the homological finiteness properties of the 
free abelian Galois covers of a toric complex. In the case 
of cyclic covers, we also study the jumping loci and the 
BNS invariants of the corresponding ``Artin kernels.''

\subsection{Homological finiteness of toric covers}
\label{subsec:htc}

We start with an elementary lemma.

\begin{lemma}
\label{lem:tran}
Let $\nu\colon \Z^n \surj \Z^r$ and $\mu\colon \Z^n \surj \Z^s$ 
be two epimorphisms, and let $\k$ be an infinite field. 
The following are equivalent:
\begin{enumerate}
\item \label{z1}
$\nu^*(\Hom(\Z^r,\k^{\times}))\cap \mu^*(\Hom(\Z^s,\k^{\times}))$ 
is finite.
\item  \label{z2}
$\nu^*(H^1(\Z^r,\Q))\cap \mu^*(H^1(\Z^s,\Q))=\{0\}$.
\end{enumerate}
\end{lemma}

\begin{proof}
Write $A=\ker(\nu)$ and $B=\ker(\mu)$. Then, 
\[
\nu^*(\Hom(\Z^r,\k^{\times}))\cap \mu^*(\Hom(\Z^s,\k^{\times}))  = 
\Hom(\Z^n/(A+B),\k^{\times})
\]
is finite if and only if the group $\Z^n/(A+B)$ is finite
(since $\k$ is infinite). 
In turn, this is equivalent to $\Hom(\Z^n/(A+B),\Q)=\{0\}$, or 
$\nu^*(\Hom(\Z^r,\Q))\cap \mu^*(\Hom(\Z^s,\Q))=\{0\}$. 
\end{proof}

Now let $L$ be a simplicial complex on vertex set $\sV$, 
with $n=\abs{\sV}$.  
\begin{theorem}
\label{thm:toric cov}
Let $\nu\colon G_L\surj \Z^r$ be an epimorphism,  
$\k$ a field, and $\nu^*\colon H^1(\Z^r,\Q) \to H^1(G_L,\Q)$ the 
induced homomorphism. For each $q\ge 0$,
the following are equivalent:
\begin{romenum}
\item  $\dim_{\k} H_{i} (T_L^{\nu}, \k) <\infty$, for all $i\le q$. 
\item $\im (\nu^*) \cap \Q^{\sW}  =\set{0}$, for all $\sW\subseteq \sV$ 
such that $\beta_{i}(\SR,z_\sW)\ne 0$, for some $i\le q$. 
\end{romenum}
\end{theorem}

\begin{proof}
Without loss of generality, we may assume $\k=\overline{\k}$. 
By Corollary \ref{cor:df cv} and Theorem \ref{thm:cv toric}\eqref{tcv},
the dimension of $H_{\le q} (T_L^{\nu}, \k)$ is finite if and only if 
the set 
\begin{equation}
\label{eq:nu toric}
\nu^*(\Hom(\Z^r,\k^{\times})) \cap  \bigg(\bigcup_{i\le q} 
\bigcup_{\stackrel{\sW\subseteq \sV}{\beta_{i}(\SR,z_\sW)\ne 0}}  
(\k^{\times})^{\sW} \bigg)
\end{equation}
is finite. 
Each coordinate subtorus $(\k^{\times})^{\sW}$ may be viewed 
as the pullback of $\Hom(\Z^s,\k^{\times})$ via a suitable 
epimorphism $\mu\colon \Z^n \surj \Z^s$; in this case, 
we also have $\Q^{\sW}= \mu^*(H^1(\Z^s,\Q))$.   The 
claim now follows from Lemma \ref{lem:tran}. 
\end{proof}

In view of Theorem \ref{thm:toric cov} and formula \eqref{eq:aah}, 
the problem of deciding whether the Betti numbers of $T_L^{\nu}$ 
(with coefficients in a field $\k$) are finite, up to a certain degree, 
reduces to a purely combinatorial and linear algebra question. 

In characteristic $0$, Theorem \ref{thm:toric cov} together 
with Theorem \ref{thm:cv toric}\eqref{tres} yield the 
following corollary. 

\begin{corollary}
\label{cor:tcq}
For each $q\ge 0$,
\[
\dim_{\Q} H_{\le q} (T_L^{\nu}, \Q) <\infty \same 
\nu^*(H^1(\Z^r,\Q)) \cap \Big(\bigcup_{i\le q} 
\RR^i_1(T_L,\Q) \Big) =\set{0}.
\]
\end{corollary}

Theorem \ref{thm:toric cov} also leads to the following qualitative 
result. 
\begin{corollary}
\label{cor:tc open}
For all integers $r\ge 1$ and $q\ge 0$, and all 
coefficient fields $\k$, the set 
\[
\{ \nu\colon G_L\surj \Z^r \mid 
\dim_{\k} H_{\le q} (T_L^{\nu}, \k) <\infty\}
\]
is a Zariski open subset of $\Grass_r(H^1(T_L,\Q))$. 
\end{corollary}

When $r>1$, this result stands in marked contrast 
with the case of an arbitrary CW-complex $X$; 
see Remark \ref{rem:non-open}.

\subsection{Artin kernels}
\label{subsec:ak}

Consider now an epimorphism $\chi\colon G_{L}\surj \Z$. 
We then have an infinite cyclic (regular) cover, 
$T_L^{\chi} \to T_{L}$. The fundamental group 
$N_{\chi}:=\pi_1(T_L^{\chi})=\ker (\chi)$ 
is called the {\em Artin kernel}\/ associated to $\chi$.  
Clearly, a classifying space for this group is the space 
$T_{\Delta}^{\chi}$, where $\Delta=\Delta_{\G}$ is the 
flag complex of $\G=L^{(1)}$. 

The most basic example is provided by the ``diagonal" 
homomorphism, $\nu\colon G_L\surj \Z$, given by 
$\nu(v)=1$, for all $v\in \V$.  The corresponding 
Artin kernel, $N_{\G}:=N_{\nu}$, is called the 
{\em Bestvina--Brady group} associated to $\G$.  
As shown in \cite{BB}, the homological finiteness 
properties of the group $N_\G$ are intimately connected 
to the connectivity properties of the flag complex $\Delta$. 
For example, $N_{\G}$ is finitely generated if and only 
if $\G$ is connected; and $N_{\G}$ is finitely presented 
if and only if $\Delta$ is simply-connected.

In \cite[Proposition 6.1]{PS-bb}, we showed that all finitely presented 
Bestvina--Brady groups are $1$-formal. In Lemma 9.1, Theorem 10.1, 
and Corollary 10.2 from \cite{PS-toric}, we extended this result to 
all Artin kernels satisfying certain homological conditions, as follows. 

\begin{theorem}[\cite{PS-toric}]
\label{thm:formal ak}
Let $\chi \colon G_{\G}\surj \Z$ be an epimorphism, 
and $N_{\chi}=\ker(\chi)$ the corresponding Artin kernel. 
\begin{romenum}
\item \label{f1}
If $H_1(N_{\chi},\Q)$ is a trivial $\Q\Z$-module, then 
$N_{\chi}$ is finitely generated. 
\item \label{f2}
If both $H_1(N_{\chi},\Q)$ and $H_2(N_{\chi},\Q)$ 
have trivial $\Q\Z$-action, then $N_{\chi}$ is $1$-formal. 
\end{romenum}
In particular, if $\G$ is a connected graph, and 
$H_1(\Delta_\G,\Q)=0$, then $N_{\G}$ is $1$-formal. 
\end{theorem}

When condition \eqref{f1} is satisfied, we also showed 
in \cite[Lemma 9.1]{PS-toric} that the inclusion map 
$\iota\colon N_{\chi} \inj G_{L}$ induces a monomorphism 
on abelianizations; in particular, $H_1(N_{\chi},\Z)$ is free 
abelian (of rank $\abs{\V}-1$).

As explained in \cite[Theorem 6.2]{PS-toric}, the $\Q\Z$-triviality 
conditions from \eqref{f1} and \eqref{f2} above can be tested 
combinatorially.  For instance, when $\chi=\nu$ is the 
Bestvina-Brady character, $H_{1}(N_{\G},\Q)$ has trivial 
$\Q\Z$-action if and only if $\G$ is connected, while 
$H_{\le 2}(N_{\G},\Q)$ has trivial $\Q\Z$-action 
if and only if $\widetilde{H}_{\le 1}(\Delta_{\G}, \Q)=0$.

\subsection{Jumping loci of Artin kernels}
\label{subsec:cjl ak}

In \cite[Lemma 8.3]{PS-bb}, we established a relationship 
between the degree $1$ (co)homology jumping loci 
of a finitely presented Bestvina-Brady group, and 
the jumping loci of the corresponding right-angled 
Artin group.   We now generalize this result to Artin kernels.  
To avoid trivialities, we will assume for the rest of this 
section that $\abs{\V}>1$.  % (i.e., $N_{\chi}\ne \{1\}$). 

\begin{theorem}
\label{thm:cjl ak}
Let $\chi\colon G_{L}\surj \Z$ be an epimorphism, 
$N_{\chi}=\ker(\chi)$, and $\iota\colon N_{\chi} \inj G_{L}$ 
the inclusion map. 
\begin{romenum}
\item \label{cvbb}
Suppose $H_1(N_{\chi},\Q)$ 
has trivial $\Q\Z$-action, and $\k=\overline{\k}$. Then
the induced homomorphism 
$\iota^{*}\colon \Hom(G_{L},\k^{\times})\to 
\Hom(N_{\chi},\k^{\times})$ 
restricts to a surjection,  
$\iota^* \colon \VV^1_1(G_{L},\k)\surj \VV^1_1(N_{\chi},\k)$.   
\item \label{resbb}
Suppose $H_i(N_{\chi},\Q)$ 
has trivial $\Q\Z$-action, for $i=1,2$. Then
the induced homomorphism 
$\iota^{*}\colon \Hom(G_{L},\C)\to \Hom(N_{\chi},\C)$ 
restricts to a surjection,  
$\iota^* \colon \RR^1_1(G_{L},\C)\surj \RR^1_1(N_{\chi},\C)$.  
\end{romenum}
\end{theorem}

\begin{proof}
Write $G=G_{L}$ and $N=N_{\chi}$. We know from 
Theorem \ref{thm:formal ak}\eqref{f1} that $N$ is finitely 
generated. 

\eqref{cvbb} 
By \cite[Lemma 9.1]{PS-toric}, $N'=G'$; hence, $N''=G''$. 
Therefore, the map $\iota\colon N\inj G$ induces 
an isomorphism $N'/N'' \isom G'/G''$ between the 
respective Alexander invariants. Taking supports 
(over $\k$), and using Corollary \ref{cor:vw} finishes 
the proof. 

\eqref{resbb} 
By \cite[Theorem 9.5]{PS-toric}, the map $\h(N, \C)\to \h(G, \C)$ 
induced by $\iota$ on holonomy Lie algebras restricts to an 
isomorphism $\h'(N, \C)\isom \h'(G, \C)$.  Abelianizing, we 
obtain an isomorphism $\h'(N, \C)/\h''(N, \C) \isom \h'(G, \C)/\h''(G, \C)$. 
Taking supports (over $\C$), and using Remark \ref{rem:holo res} 
finishes the proof.
\end{proof}

Denote by $\V_0 (\chi):= \set{v\in \V \mid \chi(v)\ne 0}$ the 
support of $\chi$.  For a subset $\sW\subset \sV$, denote by 
$\G_{\sW}$ the subgraph induced by $\G$ on $\sW$. 

\begin{corollary}
\label{cor=vrchi}
With notation as above,
\begin{romenum}
\item \label{vr1}
Suppose the module $H_{1}(N_{\chi},\Q)$ is $\Q\Z$-trivial, 
and $\k=\overline{\k}$. Then $\VV^1_1(N_{\chi},\k)$ is a union 
of subtori of the algebraic torus $(\k^{\times})^{\sV}$:
\[\
\VV^1_1(N_{\chi},\k)= \bigcup_{\stackrel{\sW\subset \sV}{
\G_{\sW}\: \textup{disconnected}}} \iota^{*}((\k^{\times})^{\sW}).
\]
If $\V_0 (\chi)\subseteq \sW$, then 
$\dim \iota^{*}((\k^{\times})^{\sW})=\abs{\sW}-1$; otherwise, 
$\dim \iota^{*}((\k^{\times})^{\sW})=\abs{\sW}$.
\item \label{vr2}
Suppose the module $H_{\le 2}(N_{\chi},\Q)$ is $\Q\Z$-trivial. Then
$\RR^1_1(N_{\chi},\C)$ is a union of rationally defined linear 
subspaces of the form $\iota^{*}(\C^{\sW})$, with the union indexed 
as in Part \eqref{vr1}, and the dimensions computed likewise.
\end{romenum}
\end{corollary}

\begin{proof}
Combine Theorems \ref{thm:cjl ak} and \ref{thm:cv toric}, 
and use formula \eqref{eq:aah}.
\end{proof}

\subsection{The BNS invariant of an Artin kernel}
\label{subsec:bns bb}
We now give a cohomological ``upper bound'' for the 
Bieri--Neumann--Strebel invariant of $N_{\chi}$.

\begin{theorem}
\label{prop:bns ak}
Suppose $H_1(N_{\chi},\Q)$ has trivial $\Q\Z$-action. 
Then
\[
\Sigma^1(N_{\chi})\subseteq \Big( \bigcup_{\sW} 
\iota^{*}\big(\R^{\sW}\big)\Big)^{\compl},
\]
where $\sW$ runs through all subsets of $\V$ such that 
$\G_{\sW}$ is disconnected.  If, in addition, $H_2(N_{\chi},\Q)$ 
is $\Q\Z$-trivial, the upper bound coincides with
$\RR^1_1(N_{\chi}, \R)^{\compl}$.
\end{theorem}

\begin{proof}
By Theorem \ref{thm:formal ak}\eqref{f1}, the group $N_{\chi}$ 
is finitely generated. By Theorem \ref{thm:main}, the set 
$\Sigma^1(N_{\chi})$ is contained in the complement of 
$\tau_1^{\R}( \VV^1_1(N_{\chi}, \C))$. 
Using Proposition \ref{prop:tctc}\eqref{tc2} and 
Corollary \ref{cor=vrchi}\eqref{vr1}, we find that 
\[
\tau_1^{\R}\big( \VV^1_1(N_{\chi}, \C)\big)=
\bigcup_{\sW} TC_1^{\R} \big( \iota^{*} ((\C^{\times})^{\sW})\big)=
\bigcup_{\sW} \iota^{*}(\R^{\sW}),
\]
and this verifies our first claim. The second claim follows from 
Corollary \ref{cor=vrchi}\eqref{vr2}.
\end{proof}

Under suitable assumptions, the above upper bound may 
be improved to a precise computation of $\Sigma^1(N_{\chi})$.

\begin{corollary}
\label{cor:bns art}
Assume $H_1(N_{\chi},\Q)$ is $\Q\Z$-trivial.  
Then $\bigcup_{\sW} \iota^{*}(\R^{\sW})= H^1(N_{\chi}, \R)$
if and only if there is a cut vertex $v\in \V$ such that $\chi(v) \ne 0$.
When this holds, $\Sigma^1(N_{\chi})=\emptyset$.
\end{corollary}

\begin{proof}
Since $H_1(N_{\chi},\Q)$ is $\Q\Z$-trivial,  Theorem 6.2 
from \cite{PS-toric} implies that the graph $\G$ is connected. 
The first assertion is then an easy consequence of 
Corollary \ref{cor=vrchi}, while the second assertion 
follows from Theorem \ref{prop:bns ak}.
\end{proof}

\begin{corollary}
\label{cor:bns bb}
If $\G$ is connected, and has a cut vertex, then 
$\Sigma^1(N_{\G})=\emptyset$.
\end{corollary}

We conclude this section by constructing an 
example of a finitely generated ($1$-formal) group 
which admits no finite presentation, yet for which the 
BNS invariant can be computed exactly. 

\begin{example}
\label{ex:rp2 again}
Let $K$ be a flag triangulation of the real projective plane, 
$\RP^2$, and let $L=K\cup e$ be the simplicial complex 
obtained from $K$ by adding an edge $e=\{v,w\}$, with 
$v$ a vertex of $K$, and $w$ a new vertex. Clearly, 
$L$ is also a flag complex; write $\G=L^{(1)}$, so that 
$L=\Delta_{\G}$.  The vertex $v$ is a cut 
point for the graph $\G$. Let $\chi$ be the character of $G_{\G}$ 
sending the vertices of $K$ to $1$, and the new vertex to $p$, 
where $p$ is $1$ or a prime number. 

By Theorem 6.2 from \cite{PS-toric}, the $\Q\Z$-module 
$H_{\le 2}(N_{\chi},\Q)$ is trivial.  Thus, by 
Theorem \ref{thm:formal ak}, the group $N_{\chi}$ is 
finitely generated and $1$-formal.  On the other hand, 
$\pi_1(L)\ne 0$, and thus, by the Main Theorem of \cite{MMV}, 
the group $N_{\chi}$ is not finitely presentable.  Finally, by 
Corollary \ref{cor:bns art}, $\Sigma^1(N_{\chi})=\emptyset$.
\end{example}

\section{K\"{a}hler and quasi-K\"{a}hler manifolds}
\label{sec:kahler}

In this final section, we discuss the relationship between 
the $\Sigma$-invariants and the homology jumping loci 
of K\"{a}hler and quasi-K\"{a}hler manifolds and their 
fundamental groups. 

\subsection{Formality}
\label{subsec:formal kahler}

A compact, connected, complex manifold $M$ is called 
a {\em K\"{a}hler manifold}\/ if $M$ admits a Hermitian 
metric for which the imaginary part is a closed $2$-form.   
A manifold $X$ is called {\em quasi-K\"{a}hler}\/ if 
$X=M\setminus D$, where $M$ is a compact K\"{a}hler 
manifold and $D$ is a divisor with normal crossings. 

The Hodge structure on $H^*(M,\C)$ puts strong constraints 
on the topology of a compact K\"{a}hler manifold $M$. 
For example, a basic result of Deligne, Griffiths, Morgan, 
and Sullivan \cite{DGMS} guarantees that $M$ must be formal.  
Consequently, the fundamental group of $M$ must be $1$-formal. 

For a quasi-K\"{a}hler manifold $X$,  the mixed 
Hodge structure $(W_{\bullet}, F^{\bullet})$ on $H^*(X,\C)$, as  
defined by Deligne \cite{D}, puts its own constraints on the 
topology of $X$. Of course, these constraints are weaker: 
for example, $\pi_1(X)$ need not be $1$-formal. Nevertheless, 
if $W_1(H^1(X,\C))=0$, then $\pi_1(X)$ is $1$-formal, 
as shown by Morgan \cite{M}.  This vanishing property holds 
whenever $X$ admits a non-singular compactification 
$\overline{X}$ with $b_1(\overline{X})=0$, cf.~\cite{D}. 
These two facts together establish the $1$-formality 
of fundamental groups of complements of projective 
hypersurfaces. 

\subsection{Jumping loci}
\label{subsec:cjl kahler}

Foundational results on the structure of the cohomology 
support loci for local systems on smooth projective varieties, 
and more generally, on compact K\"{a}hler manifolds were 
obtained by Beauville, Green and Lazarsfeld, Simpson, and 
Campana.   A more general 
result, valid in the quasi-K\"{a}hler case, was obtained by 
Arapura \cite{Ar}.  We recall this result, in a simplified form, 
which is all we need for our purposes here. 

\begin{theorem}[\cite{Ar}]
\label{thm:arapura}
Let $X$ be a quasi-K\"{a}hler manifold, with fundamental 
group $G=\pi_1(X)$. Then, all the components of $\VV^i_d(X,\C)$ 
passing through the origin are subtori of $\T_G=\Hom(G,\C^{\times})$, 
provided one of the following conditions holds. 
\begin{alphenum}
\item \label{ar1}
$i=d=1$. 
\item  \label{ar2}
$X$ is K\"{a}hler.
\item  \label{ar3}
$W_1(H^1(X,\C))=0$.
\end{alphenum}
\end{theorem}

Proposition \ref{prop:tctc}\eqref{tc2} implies that 
$\tau_1(\VV^i_d(X,\C))=TC_1(\VV^i_d(X,\C))$, whenever 
one of the above conditions is satisfied. Moreover, by 
the remarks in \S\ref{subsec:formal kahler}, the group 
$G$ is $1$-formal, provided either  
\eqref{ar2} or \eqref{ar3} is satisfied.  In this case, 
$TC_1(\VV^1_d(G,\C))=\RR^1_d(G,\C)$, by \cite{DPS-jump}. 

\subsection{$\Sigma$-invariants}
\label{subsec:bns kahler} 

Using Theorem \ref{thm:arapura} and our results from \S\ref{sec:main}, 
we obtain the following upper bounds for the BNSR invariants  
of a quasi-K\"{a}hler manifold, and its fundamental group. 

\begin{theorem}
\label{thm:bns qk}
Let $X$ be a quasi-K\"{a}hler manifold, and 
$G=\pi_1(X)$.  Then  
\begin{romenum}
\item \label{bk1}
$\Sigma^1(G)\subseteq TC_1^{\R}(\VV^1_1(G,\C))^{\compl}$.
\item  \label{bk2} 
Suppose $X$ is K\"{a}hler, or $W_1(H^1(X,\C))=0$.  Then
$ \RR^1_1(G, \R)$ is a finite union of rationally defined linear 
subspaces, and 
$\Sigma^1(G)\subseteq \RR^1_1(G, \R)^{\compl}$. 
\item  \label{bk3}
Suppose $X$ is K\"{a}hler, or 
$W_1(H^1(X,\C))=0$.  Then $U:=TC_1^{\R}(\bigcup_{i\le q} 
\VV^i_1(X,\C))$ is a finite union of rationally defined linear 
subspaces, and $\Sigma^q(X, \Z)\subseteq U^{\compl}$, 
for all $q\ge 0$. 
\end{romenum}
\end{theorem}

\begin{proof}
Part \eqref{bk1} follows from Corollary \ref{cor:bns tc} 
and Theorem \ref{thm:arapura}, part \eqref{ar1}.

Part \eqref{bk2}  follows from Proposition \ref{prop:1formal} 
and the formality of $G$. 

Part \eqref{bk3} follows from Corollary \ref{cor:bns tc} 
and Theorem \ref{thm:arapura}, part \eqref{ar2} or \eqref{ar3}. 
\end{proof}

The assumptions from \eqref{bk2} are 
really necessary, as illustrated by the next example. 

\begin{example}
\label{ex:cx heisenberg}
Let $X$ be the complex Heisenberg manifold, defined 
as the total space of the $\C^{\times}$ bundle over 
$(\C^{\times})^2$ with Euler number $1$. 
As noted in \cite{M}, $X$ is a smooth quasi-projective variety. 
Moreover, the fundamental group $G=\pi_1(X)$ is isomorphic 
to the Heisenberg group from Example \ref{ex:heisenberg}.  
Thus,  
$\Sigma^1(G)\not\subseteq \RR^1_1(G, \R)^{\compl}$.
\end{example}

\subsection{Pencils}
\label{subsec:pencils}
Next, we determine precisely when the resonance upper bound 
for $\Sigma^1$ is attained in the case when $M$ is a compact 
K\"{a}hler manifold.  The answer is in terms of {\em pencils}, that 
is, holomorphic maps from $M$ onto smooth complex curves, 
having connected generic fibers. 

We shall need a result of Arapura (contained in the proof of 
Proposition V.1.7 from \cite{Ar}):  If $f\colon M\to C$ is a 
pencil with $\chi(C)\le 0$, then $H_1(M, \C_{f^* \rho})\cong 
H_1(C, \C_{\rho})$, for all but finitely many local systems 
$\rho\in \T_{\pi_1(C)}$. 

\begin{theorem}
\label{thm:kahler}
Let  $M$ be a compact K\"{a}hler manifold with $b_1(M)>0$, and 
$G=\pi_1(M)$.  The following are equivalent:
\begin{romenum}
\item $\Sigma^1(G)=\RR^1_1(G, \R)^{\compl}$.
\item If $f\colon M\to C$ is an elliptic pencil (i.e., if $\chi(C)=0$), 
then $f$ has no fibers with multiplicity greater than $1$.  
\end{romenum}
\end{theorem}

\begin{proof}
By a recent theorem of Delzant \cite{De1}, 
\begin{equation}
\label{eq:sigmak}
\Sigma^1(G)^{\compl}= \bigcup\, f_{\alpha}^* \big( H^1(C_{\alpha}, \R)\big),
\end{equation}
where the union is taken over those pencils $f_{\alpha}\colon M\to C_{\alpha}$
with the property that either $\chi(C_{\alpha})<0$, or $\chi(C_{\alpha})=0$ and
$f_{\alpha}$ has some multiple fiber. On the other hand, Theorem B from 
\cite{DPS-serre} together with Proposition V.1.7 from \cite{Ar} 
imply that
\begin{equation}
\label{eq:resk}
\RR^1_1(G, \R)= \bigcup\, f_{\beta}^* \big( H^1(C_{\beta}, \R) \big),
\end{equation}
where the union is taken over those pencils $f_{\beta}\colon M\to C_{\beta}$ 
with $\chi(C_{\beta})<0$. To complete the proof, then, we only 
need to check that $f^* (H^1(C, \R)) \not\subseteq \RR^1_1(G, \R)$, 
whenever $f\colon M\to C$ is an elliptic pencil. 

An easy computation shows that $H_1(C, \C_{\rho})=0$, for 
all non-trivial characters $\rho\in \T_{\pi_1(C)}$.  Now suppose 
$f_{\beta}\colon M\to C_{\beta}$ is a pencil with 
$\chi(C_{\beta})<0$.  Another computation shows that 
$H_1(C_{\beta}, \C_{\rho})\ne 0$, for all $\rho\in \T_{\pi_1(C_{\beta})}$. 
By the above result of Arapura,   
$f^* \big(\T_{\pi_1(C)}\big) \not\subseteq f_{\beta}^* 
\big(\T_{\pi_1(C_{\beta})}\big)$. Using \eqref{eq:resk}, we 
conclude that $f^* (H^1(C, \R)) \not\subseteq \RR^1_1(G, \R)$. 
\end{proof}

\begin{example}
\label{ex:riemann}
Let $C$ be a Riemann surface of genus $g\ge 1$, 
and $G=\pi_1(C)$.   We then have, for any field $\k$:
\begin{romenum}
\item  \label{surf1}
$\VV^1_d(G,\k)=(\k^{\times})^{2g}$ for $d<2g-1$, 
and $\VV^1_{2g-1}(G,\k)=\{1\}$. 
\item   \label{surf2}
$\RR^1_d(G,\k)=\k^{2g}$ for $d<2g-1$, 
and $\RR^1_{2g-1}(G,\k)=\{0\}$. 
\item  \label{surf3}
$\Sigma^{q}(G,\Z)=\emptyset$, for all $q\ge 1$, if $g>1$, and
$\Sigma^{q}(G,\Z)^{\compl} =\set{0}$, for all $q$, if $g=1$.
\end{romenum}
In particular, the resonance upper bound for 
$\Sigma^1$ is attained for all non-trivial surface groups. 
\end{example}

\begin{remark}
\label{rem:beauville}
The equality $\Sigma^1(G)=\RR^1_1(G, \R)^{\compl}$ does not 
hold for arbitrary K\"{a}hler groups.  Indeed, there is a classical 
way of constructing elliptic pencils with multiple fibers on compact 
K\"{a}hler manifolds, starting from convenient finite group actions; 
see Beauville \cite[Example 1.8]{Bea}. 
\end{remark}

\subsection{Hyperplane arrangements}
\label{subsec:hyparr}
Among complex hypersurfaces, probably the best understood 
are the unions of hyperplanes.  Let $\A$ be a finite 
collection of hyperplanes in $\C^{\ell}$, with union 
$V= \bigcup_{H\in \A} H$.  The complement of the 
arrangement, $X=\C^{\ell}\setminus  V$, is a formal space. 

The homology groups $H_*(X,\Z)$ are torsion-free. 
The cohomology ring $A=H^*(X,\Z)$ depends only 
on the combinatorics of the arrangement, as encoded 
in the intersection lattice, $L(\A)$. Thus, the resonance 
varieties, $\RR^i_d(X,\k)$, are also combinatorially 
determined, for any field $\k$. 

It follows from the work of Esnault, Schechtman, and 
Viehweg \cite{ESV} that the exponential formula \eqref{eq:exp} 
holds for any arrangement complement $X$, for all 
$i\ge 0$ and $d>0$.  By Proposition \ref{prop:tc formula}, 
the resonance varieties $\RR^i_d(X,\C)$ are finite unions 
of rationally defined linear subspaces.  

As an application of Corollary \ref{cor:bns res}, we obtain 
the following combinatorial upper bounds for the BNSR 
invariants of arrangements. 

\begin{prop}
\label{prop:bns hyp}
Let $\A$ be a hyperplane arrangement in $\C^{\ell}$, 
with complement $X$. Then
$\Sigma^q(X,\Z)\subseteq \big(\bigcup_{i\le q} 
\RR^i_1(X, \R)\big)^{\compl}$, for all $q\ge 0$.  
\end{prop}

It would be interesting to know whether the 
equality $\Sigma^1(G)=\RR^1_1(G,\R)^{\compl}$ holds 
for arrangement groups, $G=\pi_1(X)$. For a complexified real
arrangement, it is easy to see that complex conjugation in $\C^{\ell}$
restricts to a homeomorphism, $\alpha\colon X \isom X$,
with the property that $\alpha_*= - \id$ on $H_1(X, \Z)$. Thus,
$\Sigma^1 (G)=- \Sigma^1 (G)$, which is consistent with the 
symmetry property of $\RR^1_1(G, \R)^{\compl}$.

\begin{ack}
A substantial part of this work was carried out while the second author 
was visiting the Institute of Mathematics of the Romanian Academy 
in October 2007, and again in July 2008.  He thanks the Institute 
for its support and hospitality during his stays in Bucharest, Romania.
\end{ack}

\vspace{-2.5pc}
\newcommand{\arxiv}[1]
{\texttt{\href{http://arxiv.org/abs/#1}{arxiv:#1}}}
\newcommand{\doi}[1]
{\texttt{\href{http://dx.doi.org/#1}{doi:#1}}}
\renewcommand{\MR}[1]
{\href{http://www.ams.org/mathscinet-getitem?mr=#1}{MR#1}}


\begin{thebibliography}{00}

\bibitem{AAH} A.~Aramova,  L.~Avramov, J.~Herzog,
{\em Resolutions of monomial ideals and cohomology 
over exterior algebras}, Trans. Amer. Math. Soc. 
\textbf{352} (1999), no.~2, 579--594.  
\MR{1603874}

\bibitem{Ar} D.~Arapura, 
{\em Geometry of cohomology support loci for local systems 
{\rm I}}, J. Alg. Geometry \textbf{6} (1997), no.~3, 563--597.  
\MR{1487227}

\bibitem{Bea} A.~Beauville, 
{\em Annulation du $H\sp 1$ pour les fibr\'{e}s en droites plats}, 
in: Complex algebraic varieties (Bayreuth, 1990), 1--15,
Lecture Notes in Math., vol.~1507, Springer, Berlin, 1992. 
\MR{1178716}

\bibitem{BB} M.~Bestvina, N.~Brady, 
{\em Morse theory and finiteness properties of groups}, 
Invent. Math. \textbf{129} (1997), no.~3, 445--470. 
\MR{1465330}

\bibitem{Bi81}  R.~Bieri, 
{\em Homological dimension of discrete groups}, 
Second edition, Queen Mary Coll. Math. Notes, 
Queen Mary College, Dept. Pure Math., London, 1981. 
\MR{0715779}

\bibitem{Bi99} R.~Bieri, 
{\em Finiteness length and connectivity length for groups}, 
in: Geometric group theory down under, 9--22, de Gruyter, 
Berlin, 1999. 
\MR{1714837} % 2000g:20070

\bibitem{Bi07} R.~Bieri,
{\em Deficiency and the geometric invariants of a group}, 
J. Pure Appl. Alg. \textbf{208} (2007), no.~3, 951--959. 
\MR{2283437} 

\bibitem{BiG}  R.~Bieri, R.~Geoghegan, 
{\em Sigma invariants of direct products of groups}, 
preprint \arxiv{0808.0013}. 

\bibitem{BGr} R.~Bieri, J.~R.~J.~ Groves,
{\em The geometry of the set of characters induced by valuations},
J. Reine Angew. Math. \textbf{347} (1984), 168--195.
\MR{0733052}  % 86c:14001 

\bibitem{BNS} R.~Bieri, W.~Neumann, R.~Strebel, 
{\em A geometric invariant of discrete groups}, 
Invent. Math. \textbf{90} (1987), no.~3, 451--477. 
\MR{0914846}

\bibitem{BR} R.~Bieri, B.~Renz,
{\em Valuations on free resolutions and higher geometric 
invariants of groups}, Comment. Math. Helvetici 
\textbf{63} (1988), no.~3, 464--497.
\MR{0960770}

\bibitem{Br} K.~S.~Brown,
{\em Trees, valuations, and the {B}ieri-{N}eumann-{S}trebel 
invariant}, Invent. Math. \textbf{90} (1987), no.~3, 479--504. 
\MR{0914847}  % 89e:20060

\bibitem{BG}  K.-U.~Bux, C.~Gonzalez,
{\em The {B}estvina-{B}rady construction revisited: 
geometric computation of $\Sigma$-invariants for 
right-angled {A}rtin groups}, J. London Math. Society 
\textbf{60} (1999), no.~3, 793--801. 
\MR{1753814}

\bibitem{Co} D.~Cohen, 
{\em Resonance of basis-conjugating automorphism groups}, 
Proc. Amer. Math. Soc. \textbf{137} (2009), no.~9, 2835--2841. 
\MR{2506439}

\bibitem{CS} D.~Cohen, A.~Suciu, 
{\em Characteristic varieties of arrangements},
Math. Proc. Cambridge Phil. Soc. \textbf{127} (1999), 
no.~1, 33--53. 
\MR{1692519} % 2000m:32036

\bibitem{D}  P.~Deligne,
{\em Th\'{e}orie de Hodge. {\rm II}},
Inst. Hautes \'{E}tudes Sci. Publ. Math. \textbf{40}
(1972), 5--57; III, ibid \textbf{44} (1974), 5--77. 
\MR{0498551}, \MR{0498552}. %58 \#16653a,b

\bibitem{DGMS}  P.~Deligne, P.~Griffiths, J.~Morgan, D.~Sullivan,
{\em Real homotopy theory of {K}\"{a}hler manifolds},
Invent. Math. \textbf{29} (1975), no.~3, 245--274.
\MR{0382702}

\bibitem{Dz} T.~Delzant,
{\em Trees, valuations and the {G}reen--{L}azarsfeld set},
Geom. Funct. Anal. \textbf{18} (2008), no.~4, 1236--1250. 
\MR{2465689} % 2009j:20035

\bibitem{De1} T.~Delzant,
{\em L'invariant de {B}ieri {N}eumann {S}trebel des groupes 
fondamentaux des vari\'{e}t\'{e}s k\"{a}hl\'{e}riennes}, preprint 
\arxiv{math.DG/0603038}, to appear in Math. Annalen.

\bibitem{DM} A.~Dimca, L.~Maxim,
{\em Multivariable {A}lexander polynomials of hypersurface 
complements}, Trans. Amer. Math. Soc. \textbf{359} (2007), 
no.~7, 3505--3528.
\MR{2299465}  %2008h:32038

\bibitem{DPS-codone} A.~Dimca, S.~Papadima, A.~Suciu,
{\em Alexander polynomials: {E}ssential variables and multiplicities}, 
Int. Math. Res. Notices \textbf{2008}, no.~3, Art. ID rnm119, 36 pp. 
\MR{2416998}  % 2009i:32036

\bibitem{DPS-serre} A.~Dimca, S.~Papadima, A.~Suciu,
{\em Formality, {A}lexander invariants, and a question 
of {S}erre}, preprint \arxiv{math.AT/0512480}.

\bibitem{DPS-jump} A.~Dimca, S.~Papadima, A.~Suciu,
{\em Topology and geometry of cohomology jump loci}, 
Duke Math. Journal \textbf{148} (2009), no.~3, 405--457.
\MR{2527322}

\bibitem{DS} A.~Dimca,  A.~Suciu,
{\em Which $3$-manifold groups are {K}\"{a}hler groups?}, 
J. Eur. Math. Soc. \textbf{11} (2009), no.~3, 521--528.
\MR{2505439}  % 2009i:32036

\bibitem{DF} W.~G.~Dwyer, D.~Fried,
{\em Homology of free abelian covers. \textup{I}},
Bull. London Math. Soc. \textbf{19} (1987), no.~4, 350--352.
\MR{0887774}

\bibitem{ESV} H.~Esnault, V.~Schechtman, E.~Viehweg,
{\em Cohomology of local systems of the complement of 
hyperplanes}, Invent. Math. \textbf{109} (1992), no.~3, 
557--561; Erratum, ibid. \textbf{112} (1993), no.~2, 447.
\MR{1176205}, \MR{1213111} %93g:32051, 94b:32061

\bibitem{Fa} M.~Farber, 
{\em Topology of closed one-forms}, 
Math. Surveys Monogr., vol.~108, Amer. Math. Soc., 
Providence, RI, 2004. 
\MR{2034601}  % 2005c:58023

\bibitem{FGS} M.~Farber, R.~Geoghegan, D.~Sch\"{u}tz, 
{\em Closed $1$-forms in topology and geometric group theory},
preprint \arxiv{0810.0962}

\bibitem{FSY} M.~Farber, A.~Suciu, S.~Yuzvinsky, 
{\em Mini-Workshop: Topology of closed one-forms and 
cohomology jumping loci}, Oberwolfach Reports \textbf{4} 
(2007), no.~3, 2321--2360.
\MR{2432117}

\bibitem{Fo} R.~H.~Fox, 
{\em Free differential calculus. \textup{I}. Derivation in the 
free group ring}, Ann. of Math. \textbf{57} (1953), 547--560. 
\MR{0053938} 

\bibitem{Ge} R.~Gehrke, 
{\em The higher geometric invariants for groups with 
sufficient commutativity}, Comm. Algebra \textbf{26} 
(1998), no.~4, 1097--1115. 
\MR{1612192}  % 99b:20072 

\bibitem{Hi} E.~Hironaka, 
{\em Alexander stratifications of character varieties}, Ann. 
Inst. Fourier (Grenoble) \textbf{47} (1997), no.~2, 555--583. 
\MR{1450425}

\bibitem{KM} M.~Kapovich, J.~Millson, 
{\em On representation varieties of {A}rtin groups, projective 
arrangements and the fundamental groups of smooth complex 
algebraic varieties}, Inst. Hautes \'{E}tudes Sci. Publ. Math.
\textbf{88} (1998), no. 8, 5--95.  
\MR{1733326} % 2001d:14024

\bibitem{Li}  A.~Libgober,
{\em First order deformations for rank one local systems 
with a non-vanishing cohomology}, Topology Appl. 
\textbf{118} (2002), no.~1-2, 159--168. 
\MR{1877722} % 2002m:52025

\bibitem{Mac}  A.~Macinic,
{\em Cohomology rings and formality properties of nilpotent groups},
preprint \arxiv{0801.4847}

\bibitem{MP}  A.~Macinic, S.~Papadima,  
{\em Characteristic varieties of nilpotent groups and 
applications}, in: Proceedings of the Sixth Congress 
of Romanian Mathematicians (Bucharest, 2007), 
vol.~1, 57--64, Romanian Academy, Bucharest, 2009.

\bibitem{MS}  D.~Matei, A.~Suciu,
{\em Hall invariants, homology of subgroups, and characteristic 
varieties}, Internat. Math. Res. Notices \textbf{2002} (2002), 
no.~9, 465--503.
\MR{1884468}  % 2003d:20055   

\bibitem{Me} J.~Meier,  
{\em Geometric invariants for {A}rtin groups}, Proc. 
London Math. Soc. (3) \textbf{74} (1997), no.~1, 151--173.
\MR{1416729}

\bibitem{MMV} J.~Meier, H.~Meinert, L.~VanWyk, 
{\em Higher generation subgroup sets and the 
$\Sigma$-invariants of graph groups}, Comment. 
Math. Helv. \textbf{73} (1998), no.~1, 22--44.
\MR{1610579}

\bibitem{M}  J.~W.~Morgan,
{\em The algebraic topology of smooth algebraic varieties},
Inst. Hautes \'{E}tudes Sci. Publ. Math. \textbf{48} (1978), 137--204.
\MR{0516917}  % 80e:55020

\bibitem{Paj}  A.~Pajitnov, 
{\em On the sharpness of inequalities of {N}ovikov type 
for manifolds with a free abelian fundamental group}, 
Math. USSR-Sb. \textbf{68} (1991), no.~2, 351--389.
\MR{1034426} % 91c:57040

\bibitem{Pa}  A.~Pajitnov, 
{\em Novikov homology, twisted {A}lexander polynomials, 
and {T}hurston cones}, St. Petersburg Math. J. 
\textbf{18} (2007), no.~5, 809--835 
\MR{2301045}  % 2007k:57023

\bibitem{PS-chen} S.~Papadima, A.~Suciu,
{\em Chen {L}ie algebras}, Int. Math. Res. Notices 
\textbf{2004}, no.~21, 1057--1086.   
\MR{2037049}

\bibitem{PS-artin} S.~Papadima, A.~Suciu, 
{\em Algebraic invariants for right-angled {A}rtin groups}, 
Math. Annalen \textbf{334} (2006), no.~3, 533--555. 
\MR{2207874}

\bibitem{PS-bb} S.~Papadima, A.~Suciu,
{\em Algebraic invariants for {B}estvina-{B}rady groups}, 
J. London Math. Society, \textbf{76} (2007), no.~2, 273--292.
\MR{2363416}

\bibitem{PS-toric} S.~Papadima, A.~Suciu,
{\em Toric complexes and {A}rtin kernels}, Advances in Math. 
\textbf{220} (2009), no.~2, 441--477.
\MR{2466422}

\bibitem{PS-mono} S.~Papadima, A.~Suciu,
{\em Algebraic monodromy and obstructions to formality}, 
preprint \arxiv{0901.0105}, to appear in Forum Math. (2010). 

\bibitem{Q} D.~Quillen,
{\em Rational homotopy theory}, Ann. of Math.
\textbf{90} (1969), 205--295.  
\MR{0258031}

\bibitem{Sc} D.~Sch\"{u}tz,
{\em On the direct product conjecture for sigma invariants}, 
Bull. London Math. Soc. \textbf{40} (2008), no.~4, 675--684. 
\MR{2441140}

\bibitem{Si}  J.-Cl. Sikorav,
{\em Homologie de {N}ovikov associ\'{e}e \`{a} une classe 
de cohomologie r\'{e}ele de degr\'{e} un}, Th\`{e}se Orsay, 1987. 

\bibitem{St1} J.~Stallings, 
{\em On fibering certain $3$-manifolds}, in: Topology of 
$3$-manifolds and related topics, pp. 95--100, Prentice-Hall, 
Englewood Cliffs, N.J., 1962. 
\MR{0158375}

\bibitem{St2} J.~Stallings,
{\em A finitely presented group whose $3$-dimensional 
integral homology is not finitely generated}, 
Amer. J. Math. \textbf{85} (1963), 541--543. 
\MR{0158917}

\bibitem{Su77}  D.~Sullivan,
{\em Infinitesimal computations in topology},
Inst. Hautes \'{E}tudes Sci. Publ. Math.
\textbf{47} (1977), 269--331.
\MR{0646078} 

\bibitem{Th} W.P.~Thurston, 
{\em A norm for the homology of $3$-manifolds}, 
Mem. Amer. Math. Soc., \textbf{59} (1986), no.~339.
\MR{0823443} %88h:57014

\end{thebibliography}
\end{document}